\documentclass[aos,preprint]{imsart}
       
\usepackage{graphicx,mathrsfs,epstopdf,subfigure}
\usepackage{amsmath,amsfonts,amssymb,amsthm}
\usepackage[colorlinks,citecolor=blue,urlcolor=blue]{hyperref}
\usepackage{tikz}
\usetikzlibrary{positioning}
\usepackage{bbm}
\numberwithin{equation}{section}

\startlocaldefs

\newtheorem{thm}{Theorem}[section]
\newtheorem{lem}[thm]{Lemma}
\newtheorem{prop}[thm]{Proposition}

\theoremstyle{definition}

\newtheorem{exmp}[thm]{Example}

\theoremstyle{remark}
\newtheorem{rem}[thm]{Remark}

\newcommand{\QED}{\ifhmode\unskip\nobreak\fi\quad {\rm Q.E.D.}} 

\newcommand{\cM}{\mathcal{M}}

\renewcommand{\P}{\mathbb{P}}

\newcommand{\R}{\mathbb{R}}

\newcommand{\cW}{\mathcal{W}}

\newcommand{\tr}{{\rm tr}}

\newcommand{\norm}[1]{\left\lVert#1\right\rVert}

\definecolor{applegreen}{rgb}{0.55, 0.71, 0.0}

\endlocaldefs

\begin{document}

\begin{frontmatter}
\title{Maximum Likelihood Estimation for Linear Gaussian Covariance Models}

\begin{aug}
\author{\fnms{Piotr} \snm{Zwiernik}\thanksref{t1}\ead[label=e1]{pzwiernik@berkeley.edu}},
\author{\fnms{Caroline} \snm{Uhler}\thanksref{t2}\ead[label=e2]{cuhler@mit.edu}},
\and
\author{\fnms{Donald} \snm{Richards}\thanksref{t3}\ead[label=e3]{richards@stat.psu.edu}}

\runauthor{P. Zwiernik, C. Uhler, D. Richards}

\affiliation{
University of California, Berkeley\thanksmark{t1},
Massachusetts Institute of Technology and IST Austria\thanksmark{t2},
Penn State University\thanksmark{t3}
}

\address{P.~Zwiernik\\
Department of Statistics\\
University of California, Berkeley\\
Berkeley, CA 94720, U.S.A. \\
\printead{e1}}

\address{C.~Uhler\\
IDSS and EECS Department \\
Massachusetts Institute of Technology\\ Cambridge, MA 02139, U.S.A.\\
\printead{e2}}

\address{D.~Richards\\
Department of Statistics \\
Penn State University \\
University Park, PA 16802, U.S.A. \\
\printead{e3}}
\end{aug}

\begin{abstract}
We study parameter estimation in linear Gaussian covariance models, which are $p$-dimensional Gaussian models with linear constraints on the covariance matrix.  Maximum likelihood estimation for this class of models leads to a non-convex optimization problem which typically has many local maxima. Using recent results on the asymptotic distribution of extreme eigenvalues of the Wishart distribution, we provide sufficient conditions for any hill-climbing method to converge to the global maximum. Although we are primarily interested in the case in which $n>\!\!>p$, the proofs of our results utilize large-sample asymptotic theory under the scheme $n/p \to \gamma > 1$. Remarkably, our numerical simulations indicate that our results remain valid for $p$ as small as $2$. An important consequence of this analysis is that for sample sizes $n \simeq 14 p$, maximum likelihood estimation for linear Gaussian covariance models behaves as if it were a convex optimization problem.
\end{abstract}

%
%



\end{frontmatter}

\section{Introduction}
In many statistical analyses, the covariance matrix possesses a specific structure and must satisfy certain constraints. We refer to~\cite{Pourahmadi2011} for a comprehensive review of covariance estimation in general and a discussion of numerous specific covariance matrix constraints. In this paper, we study Gaussian models with linear constraints on the covariance matrix. Simple examples of such models are correlation matrices or covariance matrices with prescribed zeros. 

Linear Gaussian covariance models appear in various applications. They were introduced to study repeated time series~\cite{andersonLinearCovariance} and are used in various engineering problems \cite{zhou2005comparison,dietrich2007robust}. In Section~\ref{sec_applications}, we describe in detail Brownian motion tree models, a particular class of linear Gaussian covariance models, and their applications to phylogenetics and network tomography.

To define Gaussian models with linear constraints on the covariance matrix, let $\mathbb{S}^p$ be the set of symmetric $p \times p$ matrices considered as a subset of $\R^{p(p+1)/2}$ with the scalar product given by $\langle A,B\rangle=\tr(AB)$, $A, B \in \mathbb{S}^p$, and denote by $\mathbb{S}^p_{\succ 0}$ the open convex cone in $\mathbb{S}^p$ of positive definite matrices. For \mbox{$v=(v_1,\ldots,v_r) \in \R^r$}, the $r$-dimensional Euclidean space, a random vector $X$ taking values in $\mathbb{R}^p$ is said to satisfy the \emph{linear Gaussian covariance model} given by \mbox{$G_0,G_1,\ldots,G_r \in \mathbb{S}^p$} if $X$ follows a multivariate Gaussian distribution with mean vector $\mu$ and covariance matrix $\Sigma_v$, denoted $X \sim \mathcal{N}_p(\mu,\Sigma_v)$, where 
$
\Sigma_v 
\;=\; G_0 + \sum_{i=1}^r v_i G_i.
$

In this paper, we make no assumptions about the mean vector $\mu$ and we always use the sample mean to estimate $\mu$. Then the linear covariance model, which we denote by $\mathcal{M}(\mathcal{G})$ with $\mathcal{G}=(G_0,\ldots,G_r)$, can be parametrized by
\begin{equation}
\label{ThetaGset}
\Theta_{\mathcal{G}} \;=\; \Big\{v=(v_1,\dots v_r)\in\R^{r} \;\Big|\; G_0 + \sum_{i=1}^r v_i G_i \in \mathbb{S}^p_{\succ 0}\Big\}.
\end{equation}
We assume that $G_1,\dots,G_r$ are linearly independent, meaning that \mbox{$\sum_{i=1}^r v_i G_i$} is the zero matrix if and only if $v=0$. This assumption ensures that the model is identifiable. We also assume that $G_0$ is orthogonal to the linear subspace spanned by $G_1,\ldots,G_r$, i.e., $\tr(G_{0}G_{i})=0$ for $i=1,\ldots,r$. Typically, $G_{0}$ is either the zero matrix or the identity matrix and the linear constraints on the diagonal and the off-diagonal elements are disjoint, hence satisfying orthogonality. Note that throughout this paper we assume that the matrices $G_0,G_1,\dots ,G_r$ are given and that $\Theta_\mathcal{G}$ is non-empty. The goal is to infer the parameters $\theta_1,\dots ,\theta_r$.

An important subclass of linear Gaussian covariance models is obtained by assuming that the $G_i$'s are positive semi-definite matrices and that the parameters $v_i$ are nonnegative. Brownian motion tree models discussed in Section~\ref{sec_applications} fall into this class; further examples were discussed in~\cite{rao1972estimation, kohli2016modeling}.

Another prominent example of a linear Gaussian covariance model is the set of all $p\times p$ correlation matrices; this model is defined by 
$$
\Theta_{\mathcal{G}} \;=\; \Big\{v\in\R^{\binom{p}{2}} \;\Big|\;  \mathbb{I}_p+ \sum_{1\leq i<j \leq p} v_{ij} (E_{ij}+E_{ji}) \in \mathbb{S}^p_{\succ 0}\Big\},
$$
where $\mathbb{I}_p$ denotes the identity matrix of size $p\times p$, and $E_{ij}$ is a $p\times p$ matrix whose $(i,j)$th entry is $1$ and all other entries are $0$, and $v = (v_{ij}: 1 \le i < j \le p) \in \R^{\binom{p}{2}}$ is an upper-triangular array. Similarly, also covariance matrices with prescribed zeros are linear Gaussian covariance models~\cite{Chaudhuri2007, drton2004,wermuthChains} and various methods have been described for learning the underlying sparsity structure in a covariance matrix~\cite{Rothman_2009, Bien_2011, Pourahmadi2011}.

Linear Gaussian covariance models were introduced in~\cite{andersonLinearCovariance}, motivated by the linear structure of covariance matrices in various time series models, and were used more recently for the analysis of repeated time series and longitudinal data~\cite{andrade86K, pourahmadi1999joint, jansson2000, wu2012covariance}.
Anderson studied maximum likelihood estimation for these models and proposed iterative procedures, such as the Newton-Raphson method~\cite{andersonLinearCovariance} and a scoring method~\cite{anderson73}, for calculating the maximum likelihood estimator (MLE) of $\Sigma_v$. 

We draw from $\mathcal{N}(\mu,\Sigma_v)$ a random sample $X_1,\dots, X_n$ and form the corresponding sample covariance matrix 
$$
S_{n}\;=\;\frac{1}{n}\sum_{i=1}^n (X_i-\bar X) (X_i-\bar X)^T,
$$
where $\bar X=n^{-1}\sum_{i=1}^{n}X_{i}$. In this paper we are interested in the setting where $n>\!\!>p$ and we assume throughout that $S_n\succ 0$, which holds with probability 1 when $n\geq p$. Up to an additive constant term, the log-likelihood function $\ell\!:\, \Theta_{\mathcal{G}}\to \R$, is given by 
\begin{equation}
\label{eq:logLike}
\ell(v)\quad=\quad -\;\frac{n}{2}\log \det \Sigma_{ v} -\frac{n}{2}{\rm tr}(S_{n}\Sigma_{v}^{-1}).
\end{equation}
The Gaussian log-likelihood function is a concave function of the inverse covariance matrix $\Sigma^{-1}$. Therefore, maximum likelihood estimation in Gaussian models with linear constraints on the inverse covariance matrix, as is the case for example for Gaussian graphical models, leads to a convex optimization problem~\cite{Dempster, my_paper}. However, when the linear constraints are on the covariance matrix, then the log-likelihood function generally is not concave and may have multiple local maxima~\cite{Chaudhuri2007, drton2004}. This complicates parameter estimation and inference considerably. 

A classic example of parameter estimation in linear Gaussian covariance models is the problem of estimating the correlation matrix of a Gaussian model, a venerable problem whose solution has been sought after for decades~\cite{Rousseeuw, Small, Stuart}. In Section~\ref{sec_3} we apply the results developed in this paper to provide a complete solution to this problem.

A special class of linear Gaussian covariance models for which maximum likelihood estimation is unproblematic are models such that $\Sigma$ and $\Sigma^{-1}$ have the same pattern, i.e., 
$$\Sigma = G_0 + \sum_{i=1}^r v_iG_i \qquad \textrm{and}\qquad \Sigma^{-1} = G_0 + \sum_{i=1}^r w_iG_i.$$
Examples of such models are Brownian motion tree models on the star tree, as given in (\ref{def_star_tree}), with equal variances. Szatrowski showed in a series of papers~\cite{Szatrowski1978, Szatrowski1980, Szatrowski1985} that the MLE for linear Gaussian covariance models has an explicit representation, i.e., it is a known linear combination of entries of the sample covariance matrix, if and only if $\Sigma$ and $\Sigma^{-1}$ have the same pattern. This is equivalent to requiring that the linear subspace  $\Theta_{\mathcal{G}}$ forms a Jordan algebra, i.e., if $\Sigma\in\Theta_{\mathcal{G}}$ then also $\Sigma^2\in\Theta_{\mathcal{G}}$~\cite{jensen1988}. Furthermore, Szatrowski proved that for this restrictive model class the MLE is the arithmetic mean of the corresponding elements of the sample covariance matrix and that Anderson's scoring method~\cite{anderson73} yields the MLE in one iteration when initiated at any positive definite matrix in the model.

In this paper, we show that for general linear Gaussian covariance models the log-likelihood function is, with high probability, concave in nature; consequently, the MLE can be found using any hill-climbing method such as the Newton-Raphson algorithm or the EM algorithm~\cite{RubinLinearEM}. To be more precise, in Section~\ref{sec_2} we analyze the log-likelihood function for linear Gaussian covariance models and note that it is strictly concave in the convex region 
\begin{equation}
\label{eq:delta}
\Delta_{2S_n}\quad:=\quad\{v\in \Theta_{\mathcal{G}} \mid 0\prec \Sigma_{v} \prec 2S_{n}\}.
\end{equation}
We prove that this region contains the true data-generating parameter, the global maximum of the likelihood function, and the least squares estimator, with high probability. Therefore, the problem of maximizing the log-likelihood function is essentially a convex problem and any hill-climbing algorithm, when initiated at the least squares estimator, will remain in the convex region $\Delta_{2S_n}$ and will converge to the MLE in a monotonically \mbox{increasing manner}. Other possible choices of easily computable starting points are discussed in Section \ref{sec:start}.

We emphasize that the region $\Delta_{2S_n}$ is contained in a larger subset of $\Theta_{\mathcal{G}}$ where the log-likelihood function is concave. This larger subset is, in turn, contained in an even larger region where the log-likelihood function is unimodal. Therefore, the probability bounds that we derive in this paper are lower bounds for the exact probabilities that the optimization procedure is well-behaved in the sense that any hill-climbing method will converge monotonically to the MLE when initiated at the least squares estimator.

In addition to our theoretical analysis of the behavior of the log-likelihood function we investigate computationally, with simulated data, the performance of the Newton-Raphson method for calculating the MLE. In Section~\ref{sec_3}, we discuss the problem of estimating the correlation matrix of a Gaussian model; as we noted earlier, this is a linear covariance model with linear constraints on the diagonal entries. In Section~\ref{sec_4}, we compare the MLE and the least squares estimator in terms of various loss functions and show that for linear Gaussian covariance models the MLE is usually superior to the least squares estimator. The paper concludes with a basic robustness analysis with respect to the Gaussian assumption in Section \ref{sec:robustness} and a short discussion.

\section{Brownian motion tree models}
\label{sec_applications}

In this section, we  describe the importance of linear Gaussian covariance models in practice. Arguably the most prominent examples of linear Gaussian covariance models are \emph{Brownian motion tree models}~\cite{felsenstein_maximum-likelihood_1973}. Given a rooted tree $T$ on a set of nodes \mbox{$V=\{1,\dots ,r\}$} and with $p$ leaves, where $p < r$, the corresponding Brownian motion tree model consists of all covariance matrices of the form 
$$
\Sigma_v = \sum_{i\in V} v_i\, e_{\textrm{de}(i)}\,e^T_{\textrm{de}(i)},
$$
where $e_{\textrm{de}(i)}\in\mathbb{R}^p$ is a $0/1$-vector with entry $1$ at position $j$ if leaf $j$ is a descendent of node $i$ and $0$ otherwise. Here, the parameters $v_i$ describe branch lengths and the covariance between any two leaves $i,j$ is the amount of shared ancestry between these leaves, i.e., it is the sum of all the branch lengths from the root of the tree to the least common ancestor of $i$ and $j$. Note that every leaf is a descendant of itself. As an example, if $T$ is a star tree on three leaves then
\begin{equation}
\label{def_star_tree}
\Sigma_v = 
\begin{pmatrix}
v_1+v_4 & v_4 & v_4\\ v_4 & v_2+v_4 & v_4\\ v_4&v_4& v_3+v_4 
\end{pmatrix},
\end{equation}
where $v_1, v_2, v_3$ parametrize the branches leading to the leaves and $v_4$ parametrizes the root branch. Hence the linear structure on $\Sigma_v$ is given by the structure of the underlying tree.

Felsenstein's original paper~\cite{felsenstein_maximum-likelihood_1973} introducing Brownian motion tree models as phylogenetic models for the evolution of continuous characters has been highly influential. Brownian motion tree models are now the standard models for building phylogenetic trees based on continuous traits and are implemented in various phylogenetic software such as PHYLIP~\cite{felsenstein_evolutionary_1981} or Mesquite~\cite{Lee_2006}. Brownian motion tree models represent evolution under drift and are often used to test for the presence of selective forces~\cite{Freckleton_2006, Schraiber_2013}. The early evolutionary trees were all built based on morphological characters such as body size~\cite{Freckleton_2006, cooper2010body}. In the 1960s molecular phylogenetics was made possible and from then onwards phylogenetic trees were built mainly based on similarities between DNA sequences, hence based on discrete characters. However, with the burst of new genomic data, such as gene expression, phylogenetic models for continuous traits are again becoming important and Brownian motion tree models are widely used also in this context~\cite{brawand2011evolution, Schraiber_2013}.

More recently, Brownian motion tree models have been applied to network tomography in order to determine the structure of the connections in the Internet~\cite{Nowak2010, Nowak2004}. In this application, messages are transmitted by sending packets of bits from a source node to different destinations and the correlation in arrival times is used in order to infer the underlying network structure. A common assumption is that the underlying network has a tree structure. Then the Brownian motion model corresponds to the intuitive model where the correlation in arrival time is given by the sum of the edge lengths along the path that was shared by the messages.

Of great interest in these applications is the problem of learning the underlying tree structure. But maximum likelihood estimation of the underlying tree structure in Brownian motion tree models is known to be NP-hard~\cite{Roch_2006}. The most popular heuristic for searching the tree space is Felsenstein's pruning algorithm~\cite{felsenstein_maximum-likelihood_1973,felsenstein_evolutionary_1981}. At each step, this algorithm computes the maximum likelihood estimator of the branch lengths given a particular tree structure. This is a non-convex optimization problem which is usually analyzed using hill-climbing algorithms with different initiation points. In this paper, we show that this is unnecessary and give an efficient method with provable guarantees for maximum likelihood estimation in Brownian motion models when the underlying tree structure is known.

\section{Convexity of the log-likelihood function}
\label{sec_2}

Given a sample covariance matrix $S_n$ based on a random sample of size $n$ from $\mathcal{N}(\mu,\Sigma^*)$, we denote by $\Delta_{2S_n}$ the convex subset (\ref{eq:delta}) of the parameter space. Consider the set 
$$
D_{2S_n} =\{\Sigma\in\mathbb{S}^p_{\succ 0} \mid 0\prec \Sigma \prec 2S_n\}.
$$
Note that the sets $D_{2S_n}$ and $\Delta_{2S_n}$ are random and depend on $n$, the number of observations. By construction, $D_{2S_n}$ contains a matrix $\Sigma_v$ (i.e., $\Sigma_v \in D_{2S_n}$) if and only if $\Delta_{2S_n}$ contains $v$ (i.e., $v \in \Delta_{2S_n}$). Therefore, we may identify $\Delta_{2S_n}$ with $D_{2S_n}$ and use the notations $\Sigma_v \in \Delta_{2S_n}$ and $v \in \Delta_{2S_n}$ interchangeably. In this section, we analyze the log-likelihood function $\ell(v)$ for linear Gaussian covariance models. The starting point is the following important result on the log-likelihood for the unconstrained model.

\begin{prop}\label{prop:HessianL}
The likelihood function $\ell:\mathbb{S}^p_{\succ 0}\to \R$ for the unconstrained Gaussian model is strictly concave in and only in the region $D_{2S_{n}}$.
\end{prop}

\begin{proof}
Let $A \in \mathbb{S}^p$. By (\ref{eq:logLike}), the corresponding directional derivative of $\ell(\Sigma)$ is
\begin{equation}\label{dirder}
\nabla_{A}\ell(\Sigma)=-\frac{n}{2}\tr(A\Sigma^{-1})+\frac{n}{2}\tr(S_{n}\Sigma^{-1}A\Sigma^{-1}).
\end{equation}
For $A,B\in \mathbb{S}^p$, it follows that 
\begin{equation}\label{dirhess}
\nabla_{B}\nabla_{A}\ell(\Sigma)=-\frac{n}{2}\tr((2S_{n}-\Sigma)\Sigma^{-1}B\Sigma^{-1}A\Sigma^{-1}).
\end{equation}
The log-likelihood is strictly concave in the region $D_{2S_{n}}$ if and only if $\nabla_{A}\nabla_{A}\ell(\Sigma)<0$ for every nonzero $A\in \mathbb{S}^{p}$ and $\Sigma \in D_{2S_n}$. Using the formula for $\nabla_{B}\nabla_{A}$ we can write
$$
\nabla_{A}\nabla_{A}\ell(\Sigma)=-\tr(\Sigma^{-1/2}A\Sigma^{-1/2}\,\Sigma^{-1/2}(2S_{n}-\Sigma)\Sigma^{-1/2}\,\Sigma^{-1/2}A\Sigma^{-1/2}).
$$
We rewrite this more compactly as $\nabla_{A}\nabla_{A}\ell(\Sigma)=-\tr(\tilde AC\tilde A)$, where $C=\Sigma^{-1/2}(2S_{n}-\Sigma)\Sigma^{-1/2}$ and $\tilde A=\Sigma^{-1/2}A\Sigma^{-1/2}$. If $\Sigma\in D_{2S_{n}}$ then  $C$ is positive definite and hence $\tilde AC\tilde A$ is positive semidefinite. Therefore, $\tr(\tilde AC\tilde A)\geq 0$ and it is zero only when $\tilde AC\tilde A=0$. Because $C\succ 0$, the former holds only if $A=0$, which proves strict concavity of $\ell(\Sigma)$.

To prove the converse suppose that $\Sigma\notin D_{2S_{n}}$. Then there exists a nonzero vector $u\in \R^{p}$ such that $u^{T}(2S_{n}-\Sigma)u\leq 0$. Let $A=\Sigma u u^{T}\Sigma $. Then
\begin{align*}
\nabla_{A}\nabla_{A}\ell(\Sigma)&=-\tr(uu^{T}(2S_{n}-\Sigma)uu^{T}\Sigma)\\
&=-(u^{T}\Sigma u)\cdot\left(u^{T}(2S_{n}-\Sigma)u\right)\geq 0,
\end{align*}
which proves that $\ell(\cdot)$ is \emph{not} strictly concave at $\Sigma$. This concludes the proof.
\end{proof}

As a direct consequence of Proposition~\ref{prop:HessianL} we obtain the following result on the log-likelihood function for linear covariance models:

\begin{prop}
\label{nablaNegDef}
The log-likelihood function $\ell\!: \Theta_{\mathcal{G}}\to \R$ is strictly concave on $\Delta_{2S_n}$. In particular, maximizing $\ell(v)$ over $\Delta_{2S_n}$ is a convex optimization problem. 
\end{prop}

Note that in general the region $\Delta_{2S_n}$ is contained in a larger subset of $\Theta_{\mathcal{G}}$ where the log-likelihood function is concave. This is the case, for example, for the linear Gaussian covariance models studied in~\cite{Szatrowski1978, Szatrowski1980, Szatrowski1985} with the same linear constraints on the covariance matrix as on the inverse covariance matrix. For this class of models the log-likelihood function is concave over the whole positive definite cone.

In the remainder of this section we obtain probabilistic conditions which guarantee that the true data-generating covariance matrix $\Sigma^*$, the global maximum ${\hat{v}}$, and the least squares estimator ${\bar{v}}$, are contained in the convex region $\Delta_{2S_n}$.

\subsection{The true covariance matrix and the Tracy-Widom law}
\label{sec_true_cov}

For $A\in\mathbb{S}^p$, we denote by $\lambda_{\min}(A)$ and $\lambda_{\max}(A)$ the minimal and maximal eigenvalues, respectively, of $A$. We denote by $\norm{A}$ the spectral norm of $A$, i.e.,
$$
\norm{A} \quad = \quad \max(|\lambda_{\min}(A)|, |\lambda_{\max}(A)|).
$$
In the sequel, we will often make use of the well-known fact that 
\begin{equation}
\label{lem:SimpleReduction}
\norm{A}<\epsilon \quad \textrm{if and only if}\quad -\epsilon\, \mathbb{I}_{p}\prec A\prec\epsilon\,\mathbb{I}_{p}.
\end{equation}

Given a sample covariance matrix $S_n$ based on a random sample of size $n\geq p$ from $\mathcal{N}(\mu,\Sigma^*)$, then $nS_{n}$ follows a Wishart distribution, $\mathcal{W}_{p}(n-1,\Sigma^*)$. If $\Sigma^*=\mathbb{I}_p$, we say that $nS_{n}$ follows a \emph{white Wishart} (or \emph{standard Wishart}) distribution. Note that $W_{n-1}: = n{\Sigma^*}^{-1/2}S_{n}{\Sigma^*}^{-1/2}\sim \mathcal{W}_{p}(n-1,\mathbb{I}_{p})$; also, the condition $2S_{n}-\Sigma^*\succ 0$ is equivalent to $W_{n-1}\succ \tfrac{n}{2}\,\mathbb{I}_{p}$, which holds if and only if $\lambda_{\min}(W_{n-1})>n/2$. Since $\Sigma^{*}\succ 0$ by assumption, we obtain the following lemma:

\begin{prop}\label{lem:AltMinimalEig}
The probability that $\Delta_{2S_n}$ contains the true data-generating covariance matrix $\Sigma^*$ does not depend on $\Sigma^*$ and is equal to the probability that $\lambda_{\min}(W_{n-1})>n/2$, where $W_{n-1}\sim \mathcal{W}_{p}(n-1,\mathbb{I}_{p})$.
 \end{prop}

It is generally non-trivial to approximate the distributions of the extreme eigenvalues of the Wishart distribution;~\cite[Section~9.7]{muirheadMVSbook} surveyed the results available and provided expressions for the distributions of the extreme eigenvalues in terms of zonal polynomial series expansions, which are challenging to approximate accurately when $p$ is large. Nevertheless, substantial progress has been achieved recently in the asymptotic scenario in which $n$ and $p$ are large and $n/p\to\gamma> 1$. It is noteworthy that these asymptotic approximations 
are accurate even for values of $p$ as small as $2$; we refer to~\cite{bai99, johnstone2001} for a review of these results. Recently, there has also been work on deriving probabilistic finite-sample bounds for the smallest eigenvalue~\cite{Koltchinskii}. Although in this section we base our analysis on asymptotic results, in Section~\ref{sec_Gaussian_violation} we explain the consequences of the finite-sample results with respect to violation of the Gaussian assumption.

In order to analyze the probability that $\Delta_{2S_n}$ contains the true covariance matrix $\Sigma^*$, we build on the recent work by Ma, who obtained improved approximations for the extreme eigenvalues of the white Wishart distribution~\cite{ma2012} 
in terms of the quantities 
\begin{eqnarray*}
\mu_{n,p}&=&\left((n-\tfrac{1}{2})^{1/2}-(p-\tfrac{1}{2})^{1/2}\right)^{2},\\ \sigma_{n,p}&=&\left((n-\tfrac{1}{2})^{1/2}-(p-\tfrac{1}{2})^{1/2}\right)\Big((p-\tfrac{1}{2})^{-1/2}-(n-\tfrac{1}{2})^{-1/2}\Big)^{1/3},
\end{eqnarray*}
and
\begin{equation}\label{eq:TauNu}
\tau_{n,p}\;=\;\frac{\sigma_{n,p}}{\mu_{n,p}},\qquad\quad \nu_{n,p}\;=\;\log(\mu_{n,p})+\frac{1}{8}\tau_{n,p}^{2}.
\end{equation}
We denote by $F$ the Tracy-Widom distribution corresponding to the Gaussian orthogonal ensemble~\cite{TracyWidom1996}.  Then, Ma proved the following theorem:

\begin{thm} {\rm \cite{ma2012}} \ 
\label{th:MuMinimal}
Suppose that $W_{n}$ has a white Wishart distribution, $\mathcal{W}_{p}(n,\mathbb{I}_{p})$, with $n>p$. Then, as $n\to \infty$ and $n/p\to \gamma\in (1,\infty)$, 
$$
\frac{\log \lambda_{\min}(W_{n})-\nu_{n,p}}{\tau_{n,p}}\quad\overset{\mathcal{L}}{\longrightarrow}\quad G,
$$
where $G(z)=1-F(-z)$ denotes the reflected Tracy-Widom distribution. 
\end{thm}

As a consequence of Ma's theorem we obtain:
\begin{eqnarray*}
\P(\lambda_{\min}(W_{n-1})>\frac{n}{2})&=&\P\left(\frac{\log \lambda_{\min}(W_{n-1})-\nu_{n-1,p}}{\tau_{n-1,p}}>\frac{\log (n/2)-\nu_{n-1,p}}{\tau_{n-1,p}}\right)\\
&\sim& 1-G\left(\frac{\log (n/2)-\nu_{n-1,p}}{\tau_{n-1,p}}\right),
\end{eqnarray*}
where, for sequences $a_{n,p}$ and $b_{n,p}$, we use the notation $a_{n,p}\sim b_{n,p}$ to denote that $a_{n,p}/b_{n,p}\to 1$ as $n\to \infty$ and $n/p\to \gamma\in (1,\infty)$. So, we conclude that
\begin{equation}
\label{eq:mainApprox}
\P(\Sigma^* \in \Delta_{2S_{n}})\quad\sim\quad F\left(\frac{\nu_{n-1,p}-\log (n/2)}{\tau_{n-1,p}}\right).
\end{equation}
Ma~\cite[Theorem 2]{ma2012} also proved that if $p$ is even then the reflected Tracy-Widom approximation in Theorem~\ref{th:MuMinimal} is of order $O(p^{-2/3})$; numerical simulations in~\cite{ma2012} suggest that this result holds also if $p$ is odd. In fact, this convergence rate holds in more general scenarios with the standard normalization without logarithms~\cite{feldheim2010}. However, our simulations show that Ma's approach using logarithms is more accurate  for small values of $p$. 

\begin{figure}[t!]
\centering
\subfigure[$p=3$]{\includegraphics[scale=0.2]{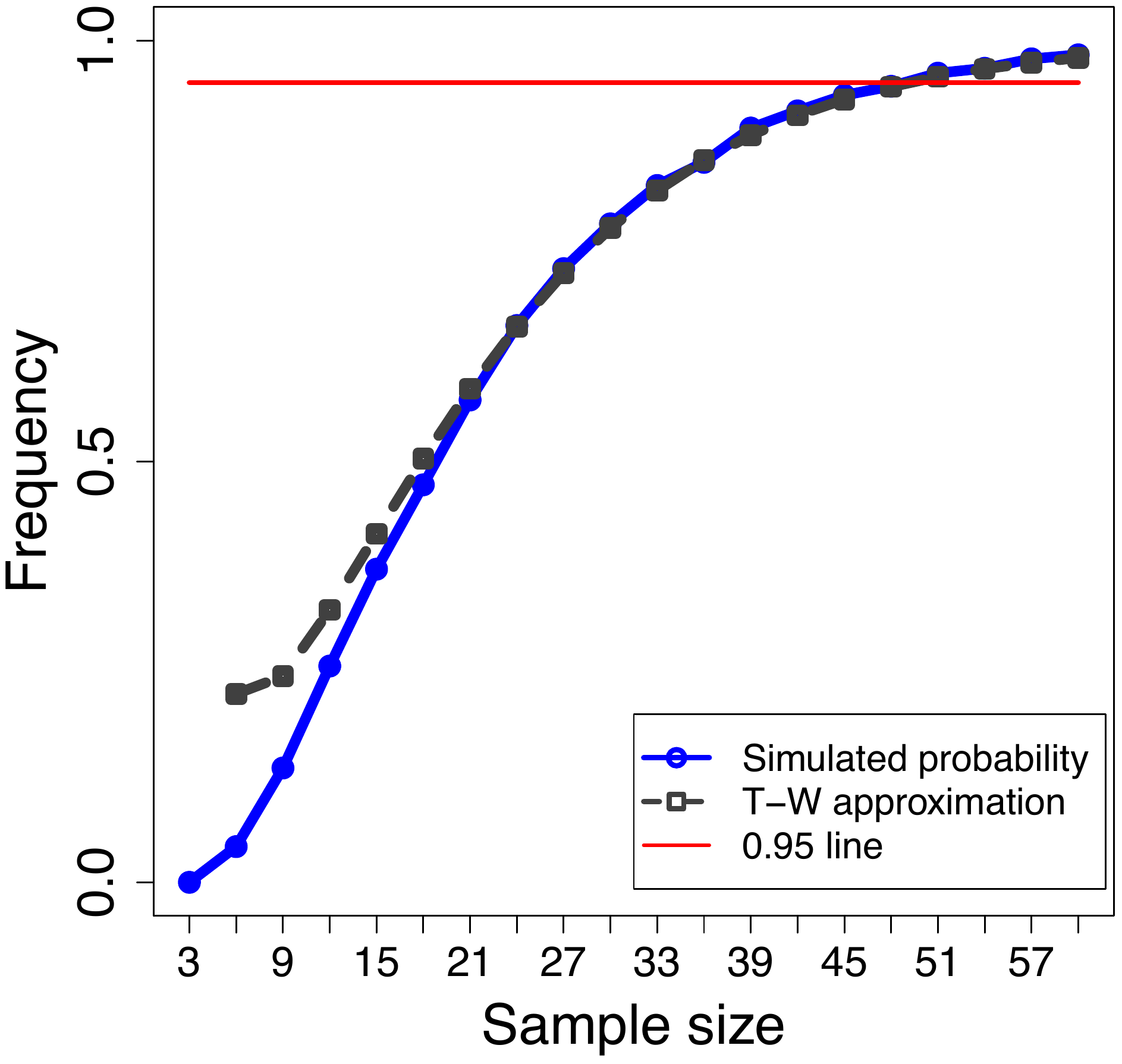}}\quad
\subfigure[$p=5$]{\includegraphics[scale=0.2]{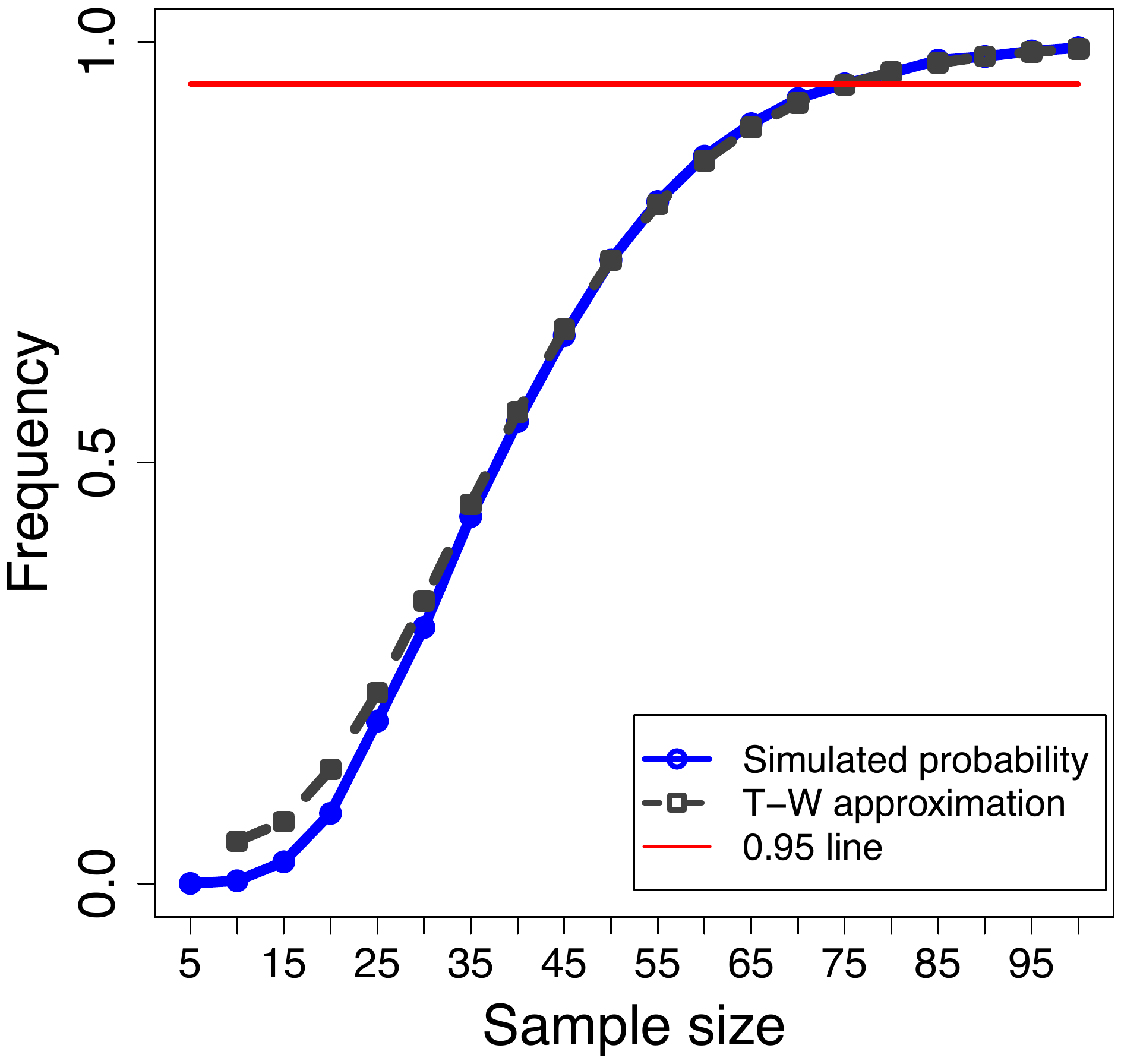}}\quad
\subfigure[$p=10$]{\includegraphics[scale=0.19]{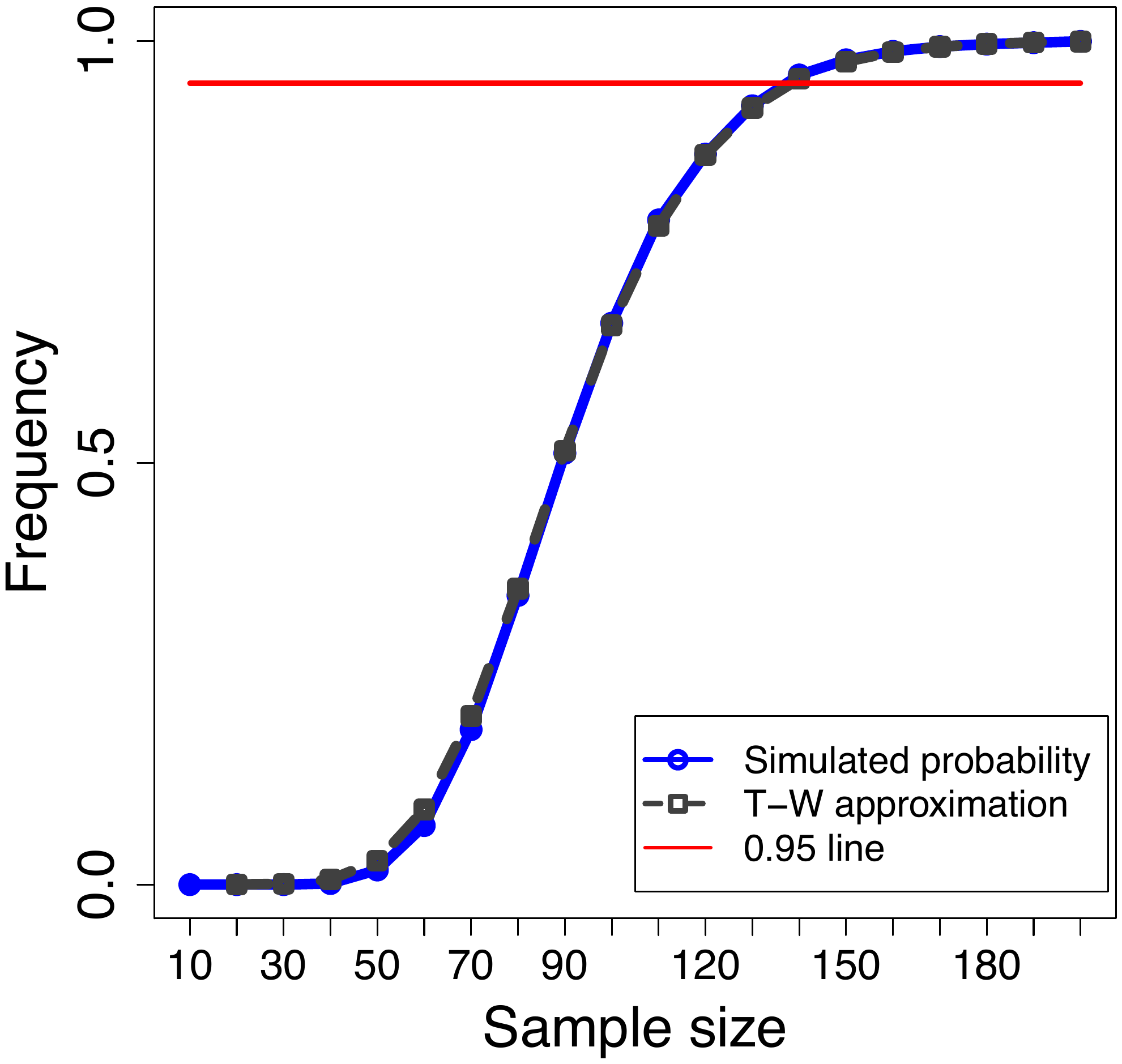}} 
\caption{In each plot, the dimension $p\in \{3,5,10\}$ is fixed and the number of samples $n$ varies between $p$ and $20p$. The solid blue line provides the simulated values of $\P(\lambda_{\min}(W_{n-1})>n/2)$, i.e., the probability that a standard Wishart random variable with $n-1$ degrees of freedom has a minimal eigenvalue that is at least $n/2$, and the dashed black line is the corresponding approximation by the reflected Tracy-Widom distribution.}
\label{fig:TW1}
\end{figure}

Another important consequence of Theorem \ref{th:MuMinimal} is that $\P(\Sigma^* \in \Delta_{2S_{n}})$ is well-approximated by a function that depends only on the ratio $n/p$. In Figure~\ref{fig:TW1}, we provide graphs from simulated data to analyze the accuracy of the approximation by the Tracy-Widom law for small $p$. By~Proposition~\ref{lem:AltMinimalEig}, $\P(\Sigma^* \in \Delta_{2S_{n}})\;=\;\P(\lambda_{\min}(W_{n-1})> \frac{n}{2})$, a quantity which does not depend on $\Sigma^*$, so we used the white Wishart distribution in our simulations. 

In each plot in Figure \ref{fig:TW1}, the dimension $p$ is fixed and the sample size $n$ varies between $p$ and $20p$. The solid blue curve is the simulated probability that $\lambda_{\min}(W_{n-1})>n/2$  and the dashed black curve is the corresponding approximation resulting from the Tracy-Widom distribution, for which we used the \texttt{R} package \texttt{RMTstat}~\cite{Johnstone:2009fj}. The values for each pair $(p,n)$ are based on $10{,}000$ replications. Figure~\ref{fig:TW1} shows that the approximation is extremely accurate even for small values of $p$. For large values of $n/p$ the two curves virtually coincide even for $p=3$. The approximation is nearly perfect over the whole range of $n$ as long as $p\geq 5$.

As an example, suppose we wish to obtain $\P(\Sigma^* \in \Delta_{2S_n}) > 0.95$. We use the approximation in (\ref{eq:mainApprox}) and the fact that $F(0.98)\approx 0.95$ to compute $n$ such that
\begin{equation}
\label{eq:ineq095}
\frac{\nu_{n-1,p}-\log(n/2)}{\tau_{n-1,p}}\;>\;0.98
\end{equation}
for given $p$. In Figure~\ref{minimalsamplesizes}, we provide for various values of $p$ the minimal sample size resulting in $\P(\Sigma^* \in \Delta_{2S_n})>0.95$. 
\begin{figure}[b!]
\caption{The minimal sample sizes resulting in $\P(\Sigma^* \in \Delta_{2S_n}) > 0.95$.}
\label{minimalsamplesizes}
\centering
\vspace{0.3cm}
\begin{tabular}{c | c c cc c cc }
$p$     &  3 & 5 &  10 &  20 &  100 &  1000  & \mbox{large}\\ \hline
$n\geq$ & 51 & 77 & 140 & 262 & 1214 & 11759 & $11.657\, p$
\end{tabular}
\end{figure}

It is interesting to study the behavior of the curves in Figure~\ref{fig:TW1} for increasing values of $p$. Our simulations show that both curves converge to the step function which is zero for ${\gamma}<\gamma^{*}$ and is one for ${\gamma}\geq \gamma^{*}$, where $\gamma^{*}$ is some fixed number. This observation suggests that for large $p$, the minimal sample size such that $\P(\Sigma^* \in \Delta_{2S_n}) > 0.95$ is equal to the minimal sample size such that $\P(\Sigma^* \in \Delta_{2S_n}) > 0$. In the following result, we formalize this observation.

\begin{prop}
\label{prop:step}
Define a sequence $(f_{n,p})_{n,p=1}^{\infty}$ of numbers by
$$
f_{n,p}\;:=\;  F\left(\frac{\nu_{n-1,p}-\log (n/2)}{\tau_{n-1,p}}\right),
$$ 
where $\nu_{n-1,p}$, $\tau_{n-1,p}$ are defined as in (\ref{eq:TauNu}) and $F$ is the Tracy-Widom distribution. Define $\gamma^{*}=6+4\sqrt{2}\approx 11.657$. Suppose that $n,p\to \infty$ and $n/p\to \gamma>1$. Then
$$
f_{n,p}\quad\to \quad \left\{\begin{array}{ll}
1, & \mbox{if } \;{\gamma}\geq \gamma^{*}\\
0, & \mbox{if } \;{\gamma}<\gamma^{*}.\end{array}
\right.
$$
\end{prop}

\begin{proof}
 Consider the expression 
$$
\frac{\nu_{n-1,p}-\log (n/2)}{\tau_{n-1,p}},
$$
and fix $n=\gamma p$, letting $p\to \infty$. Direct computation shows that the limit of this expression is $+\infty$ if $\gamma\geq \gamma^{*}$ and $-\infty$ otherwise. The result follows from $F(-\infty)=0$ and \mbox{$F(+\infty)=1$.}
\end{proof}

An alternative way to derive $\gamma^{*}=6+4\sqrt{2}$ in Proposition~\ref{prop:step} is to use the fact that $\lambda_{\min}(S_{n})$ converges almost surely to $(1-1/\sqrt{\gamma})^{2}$ \cite[Theorem 2.16]{bai99}. We have $(1-1/\sqrt{\gamma})^{2}>1/2$ if and only if $\gamma>6+4\sqrt{2}$. Proposition~\ref{prop:step} shows that the choice of the threshold $0.95$ for \mbox{constructing} Table \ref{minimalsamplesizes} is immaterial for large $p$. Moreover, note that the results in this section do not depend on the linearity of the model and can be applied to any Gaussian model where the interest lies in estimating the covariance matrix.

\subsection{The global maximum of the log-likelihood function}

We denote by $\hat{v}$ the global maximum of the log-likelihood function $\ell(v)$. The corresponding covariance matrix is denoted by $\Sigma_{\hat{v}}$. 
In this section it is convenient to work with the normalized version of the log-likelihood function
$$
\bar\ell(\Sigma)\;\;=\;\;\log\det(S_n \Sigma^{-1})-{\rm tr}(S_n\Sigma^{-1})+p,
$$
which differs from the original log-likelihood function only by a term that does not depend on $\Sigma$. By construction, $\bar\ell(S_n)=0$. Denoting by $\lambda_1\geq\cdots\geq \lambda_p$ the eigenvalues of $\Sigma^{-1/2}S_n\Sigma^{-1/2}$ and assuming that $S_n$ has full rank,  then it holds that
 $$\bar\ell(\Sigma)\;\;=\;\;\sum_{i=1}^p (\log\lambda_i-\lambda_i+1).$$ Because $\log\lambda_i-\lambda_i+1\leq 0$ for all $i=1,\ldots,p$ and it is zero only if $\lambda_i=1$, we obtain that $\bar\ell(\Sigma)<0$ for all $\Sigma\neq S_n$; so $S_n$ is the unique global maximizer. Such analysis of the Gaussian likelihood is typically attributed to \cite{watsonMLE}; see also~\cite[Section~2.2.3]{anderson1985maximum}.

Note that bounds on the value of $\bar\ell(\Sigma)$ provide bounds on the values of $\lambda_i$. In particular, if $\bar\ell(\Sigma)>-\epsilon$ for $\epsilon>0$, we obtain that $\sum_i(\lambda_i-\log\lambda_i-1)\leq \epsilon$. Since each summand is nonnegative, this implies that $0\leq \lambda_i-\log\lambda_i-1\leq\epsilon$ for each $i$. The function $f(x) = x-1-\log(x)$ is non-negative for $x>0$, strictly decreasing for $0<x<1$, reaches a global minimum at $x=1$, and is strictly increasing for $x>1$. This means, in particular, that the condition $\bar\ell(\Sigma)>-\epsilon$ provides a lower bound on $\lambda_p$ and an upper bound on $\lambda_1$. We now provide a basic result about the existence of the MLE.

\begin{prop}\label{prop:MLEexists}
Let $\mathcal{M}\subset\mathbb{S}_{\succ 0}^p$ be closed in $\mathbb{S}_{\succ 0}^p$ and non-empty. If the sample covariance matrix $S_n$ has full rank, then the MLE  over $\mathcal{M}$ exists. In particular, the MLE exists for any linear covariance model $\mathcal M(\mathcal G)$ with probability 1 when $n\geq p$.
\end{prop}
\begin{proof}
Let $c\leq 0$ be the value of the normalized log-likelihood function at some fixed point in $\mathcal{M}$. Note that the level set $\{\Sigma\in \mathbb{S}_{\succ 0}^p:\,\bar\ell(\Sigma)\geq c\}$ is compact. The constraint $\bar\ell(\Sigma)\geq c$ implies that the whole spectrum of $S_n^{1/2}\Sigma^{-1}S_n^{1/2}$ must lie in a compact subset, which implies  that $S_n^{1/2}\Sigma^{-1}S_n^{1/2}$ lies in a compact subset of $\mathbb S^p_{\succ 0}$. Since $S_n$ has full rank, $\Sigma\mapsto S_n^{1/2}\Sigma^{-1}S_n^{1/2}$ is a continuous bijection whose inverse is also continuous. This implies that $\Sigma$ is constrained to lie in a compact subset. By construction, the level set $\ell(\Sigma)\geq c$ has non-empty intersection with the closed subset $\mathcal{M}$ and hence this intersection forms a compact subset of $\mathbb{S}_{\succ 0}^p$. Then the extreme value theorem guarantees that the log-likelihood function attains its maximum, which completes the proof.
\end{proof}

In the following theorem, we apply certain concentration of measure results 
given in the appendix to obtain finite-sample bounds on the probability that the region $\Delta_{2S_n}$ contains $\Sigma_{\hat{v}}$. By construction, if $\Sigma_{\hat v}$ exists then $\Sigma_{\hat v}\succ 0$. Hence it suffices to analyze the probability that $2S_{n}-\Sigma_{\hat v}\succ 0$.

\begin{thm}
Let $\hat{v}$ be the MLE based on a full-rank sample covariance matrix $S_n$ of a linear Gaussian covariance model with parameter $v^*\in \mathbb{R}^r$ and corresponding true covariance matrix $\Sigma_{v^*}$. Let \mbox{$\epsilon = -\frac{1}{2}\log(\frac{1}{2})-\frac{1}{4} \approx 0.1$.} Then,
\begin{align*}
\P(\Sigma_{\hat{v}} \in \Delta_{2S_n})\geq 1-&\frac{2(n-1)p}{n^2\epsilon^2} -\frac{\sum\limits_{i=1}^p \psi_1(\frac{n-i}{2}) +\Big(p\log(\frac{2}{n-1})+\sum\limits_{i=1}^p \psi(\frac{n-i}{2})\Big)^2}{\epsilon^2},
\end{align*}
where $\psi(\cdot)$ and $\psi_1(\cdot)$ denote the digamma and the trigamma function, respectively. Moreover, as $n\to\infty$,
$$1-\frac{2(n-1)p}{n^2\epsilon^2} -\frac{\sum\limits_{i=1}^p \psi_1(\frac{n-i}{2}) +\Big(p\log(\frac{2}{n-1})+\sum\limits_{i=1}^p \psi(\frac{n-i}{2})\Big)^2}{\epsilon^2} = 1-\frac{4p}{n\epsilon^2} + O\Big(\frac{1}{n^2}\Big).$$ 
\end{thm}

\begin{proof}
By Proposition \ref{prop:MLEexists} the MLE $\hat v$ exists. Note that $\Sigma_{\hat v}\in \Delta_{2S_n}$ if and only if the minimal eigenvalue of $\Sigma_{\hat v}^{-1/2}S_n\Sigma_{\hat v}^{-1/2}$, denoted by $\hat\lambda_p$, is greater than $1/2$. By the discussion preceding Proposition \ref{prop:MLEexists}, if $\bar\ell(\Sigma_{\hat v})>-2\epsilon$ for $\epsilon = -\frac{1}{2}\log(\frac{1}{2})-\frac{1}{4}$, then $0\leq \hat\lambda_p-\log\hat\lambda_p-1\leq -\log{\frac{1}{2}}-\frac{1}{2}$, and in particular $\hat\lambda_p >1/2$. Since $\Sigma_{\hat v}$ is the global maximizer, this implies that $\bar\ell(\Sigma_{\hat v})\geq \bar \ell (\Sigma_{v^*})$, and in particular that
$$
\P(\bar\ell(\Sigma_{v^*})>-2\epsilon) \;\;\leq\;\; \P(\bar\ell(\Sigma_{\hat v})>-2\epsilon)\;\;\leq\;\; \P(\hat \lambda_p>1/2)\;\;=\;\; \P(\Sigma_{\hat v}\in \Delta_{2S_n}).
$$	
Therefore a lower bound on $\P(\bar\ell(\Sigma_{v^*})>-2\epsilon)$ gives a valid lower bound on $\P(\Sigma_{\hat v}\in \Delta_{2S_n})$. We have
$$
\bar\ell(\Sigma_{v^*})\;\;=\;\;\log\det(S_n\Sigma_{v^*}^{-1})-({\rm tr}(S_n\Sigma_{v^*}^{-1})-p).
$$
Noting that 
$$
\tr(S_n\Sigma_{v^*}^{-1}) \;=\; \frac{1}{n}\tr(W_{n-1}),
$$
where $W_{n-1}\sim\mathcal{W}(n-1,\mathbb{I}_p)$ {and the equality is in distribution}, we apply Chebyshev's inequality (see (\ref{eq_trace}) in the Appendix) to obtain that, with probability at least $1-\frac{2(n-1)p}{n^2\epsilon^{2}}$,
\begin{equation}
\label{event1}
-\frac{1}{n}\,p-\epsilon \quad\leq\quad \tr(S_n\Sigma_{v^*}^{-1})-p\quad\leq\quad -\frac{1}{n}\,p+\epsilon.
\end{equation}
Moreover, 
$$
\log\det (S_n\Sigma_{v^*}^{-1}) \quad =\quad \log\det W_{n-1} - p\log(n),
$$
and we again apply Chebyshev's inequality (see (\ref{eq_det}) in the Appendix) to obtain that with probability at least
\begin{align*}
p_{\det} \; :=\;1-\frac{\sum_{i=1}^p \psi_1(\frac{n-i}{2}) +\left(p\log(\frac{2}{n})+\sum_{i=1}^p \psi(\frac{n-i}{2})\right)^2}{\epsilon^2},
\end{align*}
we have 
\begin{equation}
\label{event2}
 -\epsilon\quad\leq \quad\log\det (S_n\Sigma_{v^*}^{-1})\quad\leq\quad  +\epsilon.
\end{equation}
Note that the events (\ref{event1}) and (\ref{event2}) are not independent. We then apply the inequality $\P(A\cap B)\geq \P(A)+\P(B)-1$ to deduce that with probability at least $1-\frac{2(n-1)p}{n^2\epsilon^{2}} -p_{\det}$,
$$
\bar\ell(\Sigma_{v^*})\;\;>\;\;\frac{p}{n}-\epsilon\;\;>\;\;-\epsilon,
$$
which establishes that
$$
\P(\Sigma_{\hat v}\in \Delta_{2S_n})\;\;\geq\;\; 1-\frac{2(n-1)p}{n^2\epsilon^{2}} -p_{\det}.
$$
Since $1-\frac{2(n-1)p}{n^2\epsilon^{2}} -p_{\det} $ is an increasing function of $\epsilon$, we substitute $\epsilon = -\frac{1}{2}\log(\frac{1}{2})-\frac{1}{4}$, deducing that as $n\to\infty$,
\begin{eqnarray*}
1-\frac{2p}{n\epsilon^{2}} -p_{\det} &\!\!=\!\!& 1-\frac{2p}{n\epsilon^{2}} -\frac{p\!\left(\frac{2}{n}+O(\frac{1}{n^2})\!\right)\!+\!\left(p\log(\frac{2}{n}) + p\!\left(\log(\frac{n}{2})+O(\frac{1}{n})\!\right)\!\right)^2}{\epsilon^2}\\
&\!\!=\!\!& 1-\frac{4p}{n\epsilon^{2}} +O\Big(\frac{1}{n^2}\Big),
\end{eqnarray*}
which completes the proof.
\end{proof}

More accurate finite-sample bounds for this result can be obtained by applying inequalities for the quantiles of the $\chi^2$-distribution instead of Chebyshev's inequality. Such bounds are given in Proposition \ref{prop_trace} in the Appendix.

\subsection{The least squares estimator for linear Gaussian covariance models}\label{sec:LSE}

Anderson~\cite{andersonLinearCovariance} described an unbiased estimator for the covariance matrix and recommended the use of that estimator as a starting point for the Newton-Raphson algorithm. Anderson treated the case $G_{0}=0$ and proposed as an unbiased estimator the solution to the following set of linear equations:
\begin{equation}
\label{def_eq}
\sum_{i=1}^r \;v_i\,\tr(G_i G_j) \quad = \quad \frac{n}{n-1}\tr(S_nG_j), \qquad j=1,\dots ,r.
\end{equation}
We denote by $\bar{v}$ the estimator obtained from solving (\ref{def_eq}) without the scaling by $\frac{n}{n-1}$. As we now show, $\bar{v}$ is the least squares estimator and it can be obtained by orthogonally projecting $S_n$ onto the linear subspace defining $\cM(\mathcal{G})$: Denote by $G$ the $p^{2}\times r$-matrix whose columns are $\textrm{vec}(G_1),\ldots,\textrm{vec}(G_r)$. Define
$$
H_G:=G(G^{T} G)^{-1}G^{T}.
$$
That $G^TG$ is positive definite and hence invertible follows from the assumption that $G_1,\ldots,G_r$ are linearly independent. The matrix $H_G$ is a symmetric matrix representing the projection in $\R^{p^{2}}$ onto the linear subspace spanned by $\textrm{vec}(G_1),\ldots,\textrm{vec}(G_r)$. In particular, it follows that $H_G^{2}=H_G$. The defining equations for $\bar{v}$ can be expressed as
$$
\bar{v}\;=\;(G^TG)^{-1}G^T \,\textrm{vec}(S_n),$$
and since $G_0$ is orthogonal to the linear subspace spanned by $G_1,\dots,G_r$, then 
$$
\textrm{vec}(\Sigma_{\bar v})\,=\,\textrm{vec}(G_{0})+G\, \bar v\,=\,\textrm{vec}(G_{0})+H_G\,\textrm{vec}(S_{n}).
$$

For instance, in the case of Brownian motion tree models, $G_{0}=0$ and $\Sigma_{\bar{v}}$ can be viewed as a mixture or average over all paths between `siblings' in the tree. For a star tree model with $p$ leaves, $\Sigma_{\bar{v}}$ is of the form
\begin{equation}
\label{star_tree}
(\Sigma_{\bar{v}})_{ij} \;=\; 
\begin{cases}
(S_n)_{ii}, & \hbox{ if } i = j \\
{\binom{p}{2}}^{-1} {\sum\limits_{i'<j'} (S_n)_{i'j'}}, & \hbox{ if } i < j.
\end{cases}
\end{equation}

We now prove that $\Sigma_{\bar v}$ also lies in $\Delta_{2S_n}$ with high probability. If $W_{n}$ has a standard Wishart distribution $\cW_{p}(n,\mathbb{I}_{p})$, then we can write $W_{n}=A A^{T}$, where $A$ is a $p\times n$ matrix whose entries are independent standard normal random variables. We denote the minimal and maximal singular values of $A$ by $s_{\min}(A)$ and $s_{\max}(A)$, respectively. So we have $s_{\min}(A)=\sqrt{\lambda_{\min}(W_{n})}$ and $s_{\max}(A)=\sqrt{\lambda_{\max}(W_{n})}$. We will make use of the following results from~\cite{davidson2001local} and \cite[Section 5.3.1]{vershynin2010introduction}:

\begin{thm}[\cite{davidson2001local}, Theorem II.13]
\label{th:SingBounds}
Let $A$ be a $p\times n$ matrix whose entries are independent standard normal random variables.  Then for every $t\geq 0$, with probability at least $1-2\exp(-t^{2}/2)$ we have
$$
\sqrt{n}-\sqrt{p}-t\;\leq\; s_{\min}(A)\;\leq\; s_{\max}(A)\;\leq\; \sqrt{n}+\sqrt{p}+t.
$$
\end{thm}

\begin{lem}[\cite{vershynin2010introduction}, Lemma 5.36]
\label{lem:boundSingVal}
Consider a matrix $B$ that satisfies
$$
\norm{BB^{T}-\mathbb{I}_{p}}\leq \max(\delta,\delta^{2})\qquad\mbox{for some }\delta>0.
$$
Then
$$
1-\delta\leq s_{\min}(B)\leq s_{\max}(B)\leq 1+\delta.
$$
\end{lem}

We now use these results to obtain exponential finite-sample bounds on the probability that $\Sigma_{\bar v}\in \Delta_{2S_{n}}$. Let $\kappa:=\norm{\Sigma_{v^{*}}^{-1}}\norm{\Sigma_{v^{*}}}\geq 1$ denote the condition number of $\Sigma_{v^{*}}$. 

\begin{thm}
\label{th:UNBinRegionFinite}
Let $\Sigma_{\bar{v}}$ denote the least squares estimator based on a sample covariance matrix $S_n$ of a linear Gaussian covariance model given by $\Sigma_{v^*}$. Fix $p$ and suppose that $n\geq 15$ and large enough so that 
$$
\epsilon=\epsilon(n)\;:=\;\sqrt{\frac{2+\kappa\sqrt{p}}{1+\kappa\sqrt{p}}}-\sqrt{\frac{n-1}{n}}-\sqrt{\frac{p}{n}}\;>\; 0.
$$
Then $2S_{n}-\Sigma_{\bar v}\succ 0$ with probability at least $1-2\exp(-n\epsilon^{2}/2)$. Also, $\Sigma_{\bar v}\succ 0$ with probability at least $1-2\exp(-n\epsilon^{2}/2)$. Consequently, 
$$
\P(\Sigma_{\bar v} \in \Delta_{2S_{n}})\geq 1-4\exp(-n\epsilon^{2}/2).
$$
\end{thm}

\begin{proof}
Note that $2S_{n}-\Sigma_{\bar v}\succ 0$ if and only if
$$
C\;:=\;\Sigma_{v^{*}}^{-1/2}S_{n}\Sigma_{v^{*}}^{-1/2}+\Sigma_{v^{*}}^{-1/2}(S_{n}-\Sigma_{\bar v})\Sigma_{v^{*}}^{-1/2}\;\succ\; 0,
$$
or equivalently, if $\lambda_{\min}(C)> 0$. By the eigenvalue stability inequality~\cite[Equation (1.63)]{Tao}, for $A, B \in \mathbb{S}^p$ and for any eigenvalue $\lambda_i$,
\begin{equation}
\label{eig_sta}
|\lambda_i(A+B)-\lambda_i(A)| \;\leq\; \norm{B}.
\end{equation}
As a special case of this inequality we obtain
$$\lambda_{\min}(A+B)\;\geq\; \lambda_{\min}(A) - \norm{B}.$$ 
Setting $A = \Sigma_{v^{*}}^{-1/2}S_{n}\Sigma_{v^{*}}^{-1/2}$ and $B = \Sigma_{v^{*}}^{-1/2}(S_{n}-\Sigma_{\bar v})\Sigma_{v^{*}}^{-1/2}$ in the latter inequality, we obtain 
\begin{equation}\label{eq:app:bounds1}
\begin{split}
\lambda_{\min}(C)&\geq \lambda_{\min}(\Sigma_{v^{*}}^{-1/2}S_{n}\Sigma_{v^{*}}^{-1/2})-\norm{\Sigma_{v^{*}}^{-1/2}(S_{n}-\Sigma_{\bar v})\Sigma_{v^{*}}^{-1/2}}\\
& \geq \lambda_{\min}(\Sigma_{v^{*}}^{-1/2}S_{n}\Sigma_{v^{*}}^{-1/2})-\norm{\Sigma_{v^{*}}^{-1}}\norm{S_{n}-\Sigma_{\bar v}}\\
& =\frac{1}{n}\lambda_{\min}(W_{n-1})-\norm{\Sigma_{v^{*}}^{-1}}\norm{S_{n}-\Sigma_{\bar v}},
\end{split}
\end{equation}
where $W_{n-1}:=n\Sigma_{v^{*}}^{-1/2}S_{n}\Sigma_{v^{*}}^{-1/2}\sim W(n-1,\mathbb{I}_p)$. For the second inequality in (\ref{eq:app:bounds1}) we used the inequality $\norm{ABA}\leq \norm{A}^{2}\norm{B}$ and the fact that $\norm{A}^{2}=\norm{A^{2}}$ for any positive definite matrix $A$. Since $\bar{v}$ is the least squares estimator,
$$
\norm{S_{n}-\Sigma_{\bar v}}_{F}\leq \norm{S_{n}-\Sigma_{v^{*}}}_{F}.
$$
As a consequence of standard inequalities on matrix norms we obtain
\begin{align*}
\norm{S_{n}-\Sigma_{\bar v}} &\leq \norm{S_{n}-\Sigma_{\bar v}}_{F} \\
&\leq \norm{S_{n}-\Sigma_{v^{*}}}_{F} \leq \sqrt{p}\norm{S_{n}-\Sigma_{v^{*}}}.
\end{align*}
This implies that
\begin{equation}\label{eq:bound2}
\begin{split}
\lambda_{\min}(C)&\geq \frac{1}{n}\lambda_{\min}(W_{n-1})-\sqrt{p}\norm{\Sigma_{v^{*}}^{-1}}\norm{S_{n}-\Sigma_{v^{*}}}\\
& = \frac{1}{n}\lambda_{\min}(W_{n-1})-\sqrt{p}\norm{\Sigma_{v^{*}}^{-1}}\norm{\Sigma_{v^{*}}^{1/2}\left(\frac{1}{n}W_{n-1}-\mathbb{I}_{p}\right)\Sigma_{v^{*}}^{1/2}}\\
& \geq \frac{1}{n}\lambda_{\min}(W_{n-1})-\sqrt{p}\norm{\Sigma_{v^{*}}^{-1}}\norm{\Sigma_{v^{*}}}\norm{\frac{1}{n}W_{n-1}-\mathbb{I}_{p}}.
\end{split}
\end{equation}
Define 
$$
\delta:=\frac{1}{1+\kappa\sqrt{p}}.
$$ 
Suppose that $\norm{\frac{1}{n}W_{n-1}-\mathbb{I}_{p}}\leq \delta$. Then ${\frac{1}{n}\lambda_{\min}(W_{n-1})}\geq 1-\delta > 0$, since $\delta < 1$, and hence the last line in (\ref{eq:bound2}) can be bounded further as follows:
$$
\frac{1}{n}\lambda_{\min}(W_{n-1})-\kappa\sqrt{p}\norm{\frac{1}{n}W_{n-1}-\mathbb{I}_{p}}\geq 1-\delta-\kappa\delta\sqrt{p}=0.
$$
Combining the arguments made so far in this proof, we obtain
 a lower bound on the probability that $2S_{n}-\Sigma_{\bar v}\succ 0$, or equivalently that $2S_{n}-\Sigma_{\bar v}\succeq 0$: 
$$
\P(2S_{n}-\Sigma_{\bar v}\succeq 0)\quad=\quad\P(\lambda_{\min}(C)\geq 0)\quad\geq\quad  \P\left(\norm{\frac{1}{n}W_{n-1}-\mathbb{I}_{p}}\leq \delta\right),
$$
and it therefore suffices to obtain a bound for $\P\left(\norm{\frac{1}{n}W_{n-1}-\mathbb{I}_{p}}\leq \delta\right)$. 

Note that $\norm{\frac{1}{n}W_{n-1}-\mathbb{I}_{p}}\leq \delta$ is equivalent to 
\begin{equation}\label{eq:LambdasBound}
(1-\delta)n\leq \lambda_{\min}(W_{n-1})\leq \lambda_{\max}(W_{n-1})\leq (1+\delta)n.
\end{equation}
By Theorem \ref{th:SingBounds} we have that for every $t\geq 0$ 
\begin{equation}\label{eq_1t}
\sqrt{n-1}-\sqrt{p}-t\leq \sqrt{\lambda_{\min}(W_{n-1})}\leq \sqrt{\lambda_{\max}(W_{n-1})}\leq \sqrt{n-1}+\sqrt{p}+t
\end{equation}
with probability at least $1-2\exp(-t^{2}/2)$. To bound the probability of the event (\ref{eq:LambdasBound}), we set 
\begin{equation}\label{eq_2t}
t:=\sqrt{n}\left(\sqrt{1+\delta}-\sqrt{\frac{n-1}{n}}-\sqrt{\frac{p}{n}}\right)=\sqrt{n}\epsilon(n).
\end{equation}
Since $\epsilon(n)>0$ by assumption, we also have $t\geq 0$. By substituting the expression for $t$ given in (\ref{eq_2t}) into the left inequality in (\ref{eq_1t}), we obtain
$$
\sqrt{n}\bigg(2\sqrt{\frac{n-1}{n}}-\sqrt{1+\delta}\bigg)\leq \sqrt{\lambda_{\min}(W_{n-1})}.
$$
By applying Lemma~\ref{lem:boundSingVal} we obtain 
$$
\sqrt{\lambda_{\max}(W_{n-1})}\leq \sqrt{n}\sqrt{1+\delta}.
$$
Hence, with probability at least $1-2\exp(-n\epsilon^{2}/2)$, there holds the inequalities 
$$
\sqrt{n}\bigg(2\sqrt{\frac{n-1}{n}}-\sqrt{1+\delta}\bigg)\leq \sqrt{\lambda_{\min}(W_{n-1})}\leq \sqrt{\lambda_{\max}(W_{n-1})}\leq \sqrt{n}\sqrt{1+\delta}.
$$ 
Now note that
$$
\sqrt{n}\sqrt{1-\delta}\leq \sqrt{n}\left(2\sqrt{\frac{n-1}{n}}-\sqrt{1+\delta}\right)
$$
for $n\geq 15$. Hence also, with probability at least $1-2\exp(-n\epsilon^{2}/2)$ and $n\geq 15$, we have
$$
\sqrt{n}\sqrt{1-\delta}\leq \sqrt{\lambda_{\min}(W_{n-1})}\leq \sqrt{\lambda_{\max}(W_{n-1})}\leq \sqrt{n}\sqrt{1+\delta},
$$ 
which is equivalent to (\ref{eq:LambdasBound}). This establishes the exponential bound on the probability that $2S_{n}-\Sigma_{\bar v}\succ 0$. 

The same bound holds for the event $\Sigma_{\bar v}\succ 0$, since $\Sigma_{\bar v}\succ 0$ is equivalent to 
$$
C'\;:=\;\Sigma_{v^{*}}^{-1/2}S_{n}\Sigma_{v^{*}}^{-1/2}+\Sigma_{v^{*}}^{-1/2}(\Sigma_{\bar v}-S_{n})\Sigma_{v^{*}}^{-1/2}\;\succ\; 0
$$
i.e., $\lambda_{\min}(C')>0$. Again by the eigenvalue stability inequality (\ref{eig_sta}) we have
$$
\lambda_{\min}(C')\geq \lambda_{\min}(\Sigma_{v^{*}}^{-1/2}S_{n}\Sigma_{v^{*}}^{-1/2})-\norm{\Sigma_{v^{*}}^{-1/2}(S_{n}-\Sigma_{\bar v})\Sigma_{v^{*}}^{-1/2}},
$$
the same expression as in (\ref{eq:app:bounds1}). Finally, because the events $\{\Sigma_{\bar v}\succ 0\}$ and $\{2S_{n}-\Sigma_{\bar v}\succ 0\}$ are not independent, we utilize the inequality $\P(A\cap B)\geq \P(A)+\P(B)-1$ to complete the proof.
\end{proof}

\begin{rem}
We emphasise that Theorem \ref{th:UNBinRegionFinite}	 provides also lower bounds on the probability that the least squares estimator is positive definite, which is of its own practical importance. 
\end{rem}

We now discuss some simulation results for $\mathbb{P}(\Sigma_{\bar{v}} \in \Delta_{S_{n}})$ for the example of Brownian motion tree models on the star tree.

\begin{exmp}\label{ex:simple_unb}
Consider the model $\cM(G)$, where $G_{0}$ is the zero matrix, $G_{i}=E_{ii}$ for $i=1,\ldots,p$, where $E_{ii}$ is the matrix with a $1$ in position $(i,i)$ and zeros otherwise, and $G_{p+1} = \mathbbm{1}\mathbbm{1}^{T}$, where $\mathbbm{1}$ is the column vector with every component equal to $1$. This linear covariance model corresponds to a Brownian motion tree model on the star tree. Suppose that the true covariance matrix $\Sigma_{v^*}$ is given by $v^*_{i}=i$ for $i=1,\dots ,p$ and $v^*_{p+1}=1$. We performed simulations similar to the ones that led to Figure~\ref{fig:TW1}. 

For fixed $p\in \{3,5,10\}$ we let $n$ vary between $p$ and $20p$. For each pair $(p,n)$ we generated $10,000$ times a sample of size $n$ from $\mathcal N_{p}(0,\Sigma_{v^*})$, computed the corresponding sample covariance matrix and the least squares estimator, and determined whether the least squares estimator belonged to the region $\Delta_{2S_n}$. The simulated probabilities of the event $0\prec \Sigma_{\bar v}\prec 2S_{n}$ are given by the solid blue line in Figure~\ref{fig:unb2}. As in Figure~\ref{fig:TW1}, the dashed black line represents the approximated probabilities, with the approximation being obtained through the Tracy-Widom law, of the event $0\prec \Sigma_{v^*}\prec 2S_{n}$. Figure~\ref{fig:unb2} indicates that on average $\Sigma_{\bar v}$ lies in $\Delta_{2S_n}$ more often than $\Sigma_{v^*}$. 
\end{exmp}

\begin{figure}[t!]
\centering
\subfigure[$p=3$]{\includegraphics[scale=0.2]{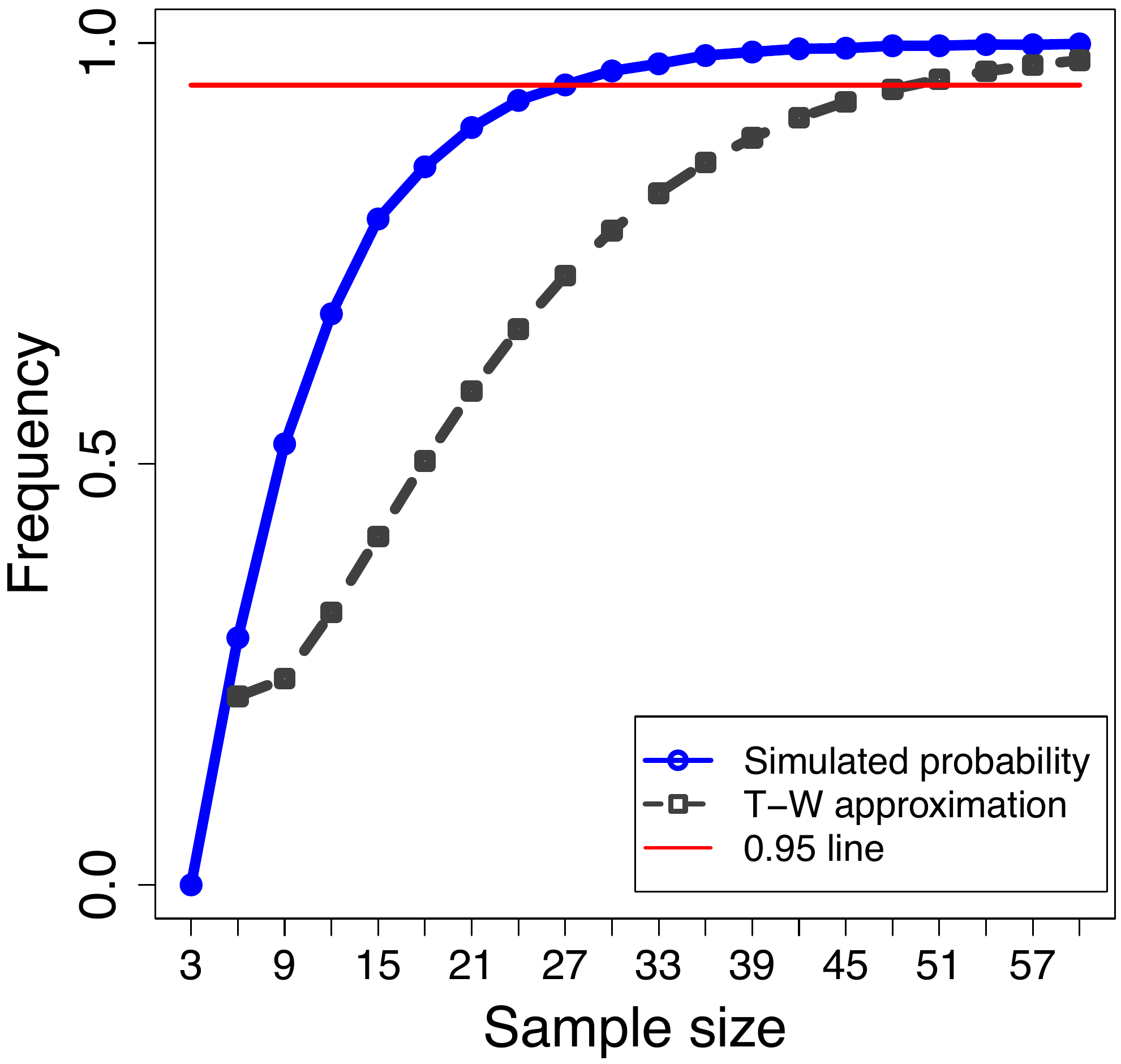}}\quad
\subfigure[$p=5$]{\includegraphics[scale=0.23]{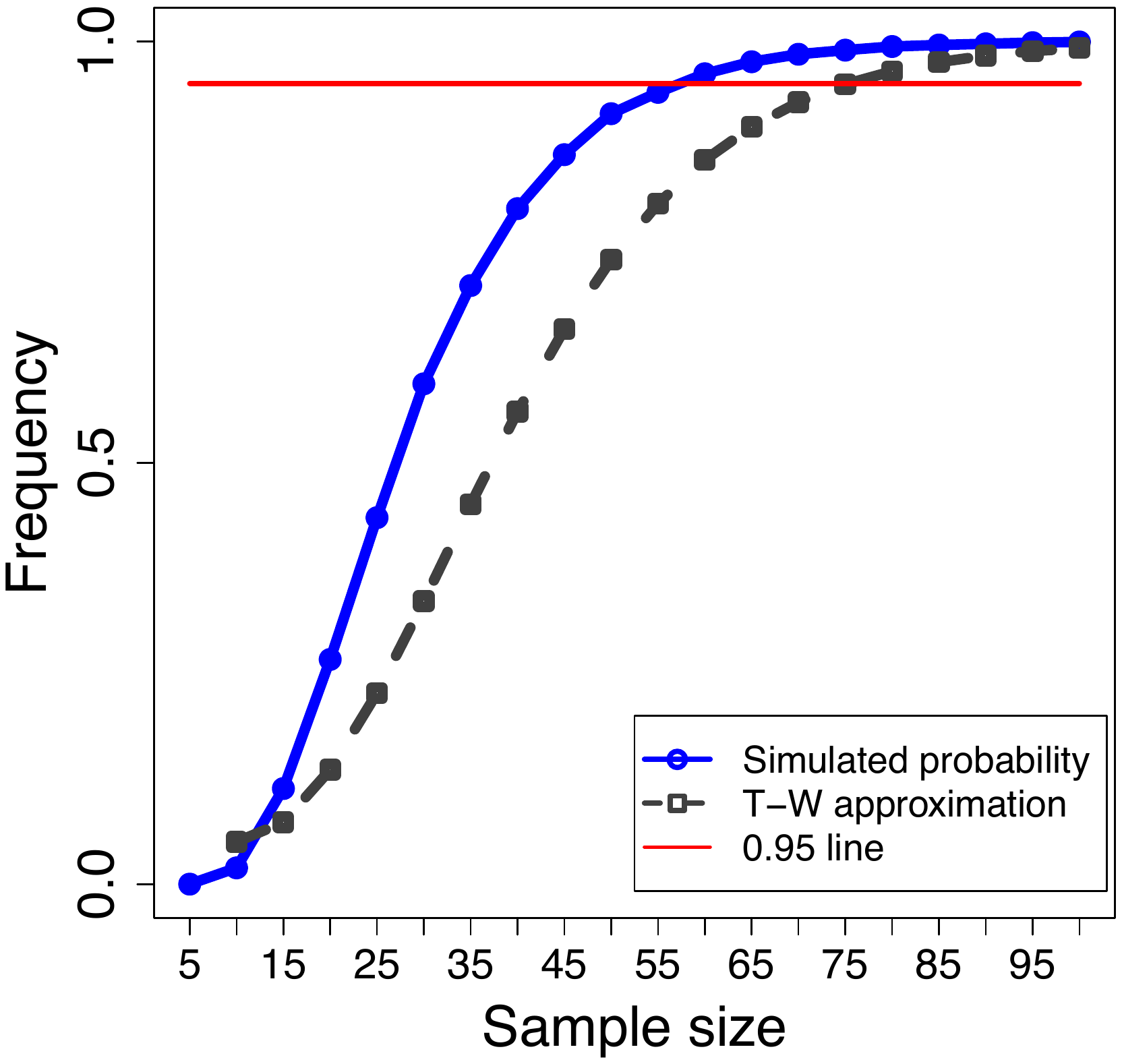}}\quad
\subfigure[$p=10$]{\includegraphics[scale=0.2]{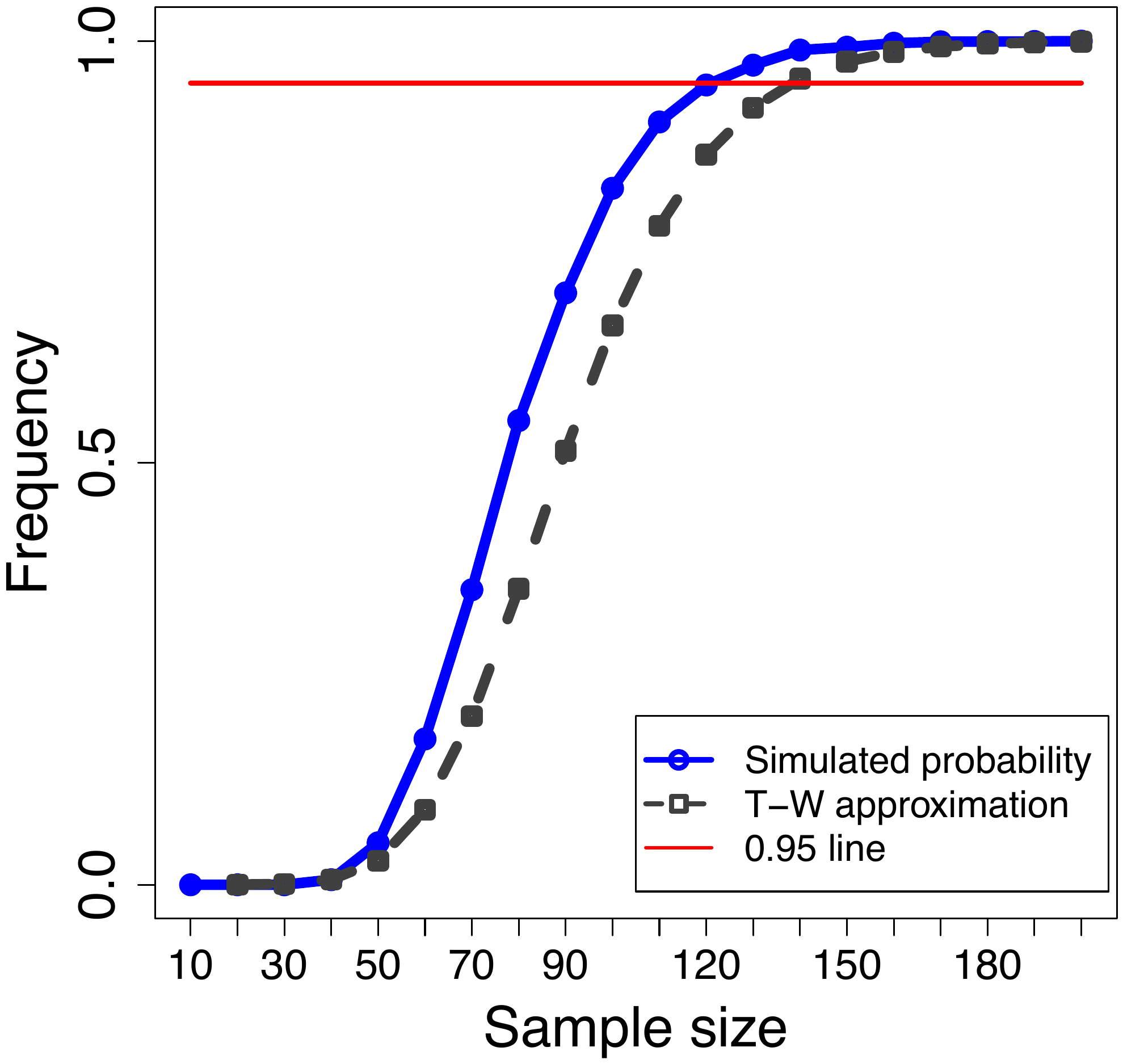}} 
\caption{Illustration of Example \ref{ex:simple_unb}. In each plot the dimension $p\in \{3,5,10\}$ is fixed and the number of samples $n$ varies between $p$ and $20p$. The solid blue curve shows the simulated probability that the least squares estimator $\Sigma_{\bar v}$ satisfies $0\prec \Sigma_{\bar v}\prec 2S_{n}$ and the dashed black curve is the approximation of the probability that the MLE $\Sigma_{v^*}$ satisfies $0\prec \Sigma_{v^*}\prec 2S_{n}$ by the reflected Tracy-Widom distribution.}
\label{fig:unb2}
\end{figure}

\subsection{Alternative initiation points for hill-climbing}\label{sec:start}

The least squares estimator $\Sigma_{\bar v}$ is a natural initiation point in any hill-climbing algorithm for computing the MLE. We showed that such an initiation comes with convergence guarantees when $\Sigma_{\bar v}$ lies in $\Delta_{2S_n}$ and we bounded from below the probability of the event $\{\Sigma_{\bar v}\in \Delta_{2S_n}\}$. In this section, we describe an alternative initiation for the case when $\Sigma_{\bar v}\notin\Delta_{2S_n}$. The suggested initiation is motivated by shrinkage estimators used for high-dimensional covariance matrix estimation~\cite{Ledoit2004365,schafer2005shrinkage}.

In the following, we analyze linear Gaussian covariance models with $G_0=0$. Similar techniques can be developed in other scenarios. Let $\Sigma_0$ be a fixed matrix and consider the convex combination $S_n^*=s((1-t)\Sigma_0+tS_n)$ of $G_0$, $\Sigma_0$ and $S_n$, where $s,t\in (0,1]$. We now show that by choosing $s$ and $t$ appropriately, the projection of $S_n^*$ onto $\mathcal{M}(G)$ lies in the region $\Delta_{2S_n}$ for any matrix $\Sigma_0\in\mathcal{M}(G)$. Hence $S_n^*$ can be used as an alternative initiation when $\Sigma_{\bar v}\notin\Delta_{2S_n}$.


\begin{prop}
\label{prop_alternative}
	Suppose that $\mathcal M(\mathcal G)$ is a linear Gaussian covariance model with $G_0=0$. Let $\Sigma_0$ be any covariance matrix in $\mathcal M(G)$ and consider $S_n^*=s((1-t)\Sigma_0+t S_n)$, where the coefficients $s,t\in (0,1]$ are such that $t=1$ if $\Sigma_{\bar v}\succ 0$ and
	$$t<\frac{\lambda_{\min}(\Sigma_0)}{\lambda_{\min}(\Sigma_0)-\lambda_{\min}(\Sigma_{\bar v})}$$
	otherwise, and $s=1$ if $2S_n-((1-t)\Sigma_0+t\Sigma_{\bar v})\succ 0$ and $$s<2\lambda_{\min}\big(S_n((1-t)\Sigma_0+t\Sigma_{\bar v})^{-1}\big)$$ otherwise. Then the orthogonal projection of $S_n^*$ onto $\mathcal M(\mathcal G)$ given by $\Sigma^*=s((1-t)\Sigma_0+t \Sigma_{\bar v})$ lies in $\Delta_{2S_n}$. 
\end{prop}

\begin{proof}
Consider the following three cases. First, if $\Sigma_{\bar v}\succ 0$ and $2S_n-\Sigma_{\bar v}\succ 0$, then $\Sigma^*=\Sigma_{\bar v}\in \Delta_{2S_n}$ and $S^*_n=S_n$ and hence the proposition holds. Second, if $\Sigma_{\bar v}\succ 0$ and $2S_n-\Sigma_{\bar v}\not\succ 0$, then $t=1$ and $\Sigma^*=s\Sigma_{\bar v}\succ 0$. Using the constraint on $s$ we obtain
$$
\lambda_{\min}(S_n(\Sigma^*)^{-1})=\lambda_{\min}(S_n(s\Sigma_{\bar v})^{-1})=\frac{1}{s}\lambda_{\min}(S_n\Sigma_{\bar v}^{-1})>\frac{1}{2},
$$
or equivalently, $2S_n-\Sigma^*\succ 0$, which implies $\Sigma^*\in \Delta_{2S_n}$. The third and final case is when $\Sigma_{\bar v}\not\succ 0$. In this case we define $\Sigma'=(1-t)\Sigma_0+t\Sigma_{\bar v}$. By concavity of the minimal eigenvalue we obtain
\begin{eqnarray*}
\lambda_{\min} ((1-t) \Sigma_0+t\Sigma_{\bar v})&\geq &(1-t)\lambda_{\min}(\Sigma_0)+t\lambda_{\min}(\Sigma_{\bar v})\\
&=&\lambda_{\min}(\Sigma_0)-t(\lambda_{\min}(\Sigma_0)-\lambda_{\min}(\Sigma_{\bar v})).
\end{eqnarray*}
Hence, $\Sigma'\succ 0$ for every $t<\frac{\lambda_{\min}(\Sigma_0)}{\lambda_{\min}(\Sigma_0)-\lambda_{\min}(\Sigma_{\bar v})}$. By replacing $\Sigma_{\bar v}$ by $\Sigma'$ in the second case it follows that $\Sigma^*=s\Sigma'\in \Delta_{2S_n}$. It remains to show that in all three cases $\Sigma^*$ is the orthogonal projection of $S^*$ onto $\mathcal M(G)$. For this, observe that if $\Sigma_0$ lies in the linear span of $\mathcal{M}(\mathcal{G})$ then the orthogonal projection of $a\Sigma_0+bS_n$ is equal to $a\Sigma_0+b\Sigma_{\bar v}$, which completes the proof.
  \end{proof}

Note that when $\Sigma_{\bar{v}}\in\Delta_{2S_n}$, then Proposition~\ref{prop_alternative} results in the least squares estimator. When $\Sigma_{\bar{v}}\succ 0$, but $\Sigma_{\bar{v}}\nsucc 2S_n$, then Proposition~\ref{prop_alternative} shrinks the estimator to lie in $\Delta_{2S_n}$. If the identity matrix lies in $\mathcal M(G)$ then it is a natural choice for $\Sigma_0$. If $\mathcal M(\mathcal G)$ contains all diagonal matrices as for example for Brownian motion tree models, then it is more natural to take $\Sigma_0$ to be the diagonal matrix with sample variances on the diagonal, that is, $\Sigma_0={\rm diag}(S_n)$. Since $\Sigma_0$ and $\Sigma_{\bar v}$ coincide on the diagonal, this estimator is obtained from $\Sigma_{\bar v}$ by shrinking off-diagonal entries by a factor of $st$ and the diagonal entries by a factor of $s$.

\section{Estimating correlation matrices}
\label{sec_3}

In this section, we discuss the problem of maximum likelihood estimation of correlation matrices. In the case of $3\times 3$ matrices this problem was analyzed by several authors~\cite{Rousseeuw, Small, Stuart}. Computing the MLE for a correlation matrix exactly requires solving a system of polynomial equations in $\binom{p}{2}$ variables and can only be done using Gr\"obner bases techniques for very small instances. In the following, we demonstrate how this problem can be solved using the Newton-Raphson method and show how such an approach performs for estimating a $3\times 3$ and a $4\times 4$ correlation matrix.

We initiate the algorithm at the least squares estimator. For correlation models with no additional structure, the least squares estimator is given by
$$\bar{v}_{ij} \;=\; (S_n)_{ij}, \qquad 1\leq i< j\leq p,$$
and the corresponding correlation matrix is
$$\Sigma_{\bar{v}} \;=\; \mathbb{I}_p+ \sum_{1\leq i< j\leq p} \bar{v}_{ij} (E_{ij} + E_{ji}),$$
where, as before, $E_{ij}$ is the matrix with a $1$ in position $(i,j)$ and zeros elsewhere. Let $v^{(k)}$ be the $k$-th step estimate of the parameter vector $v$ obtained using the Newton-Raphson algorithm. At step $k+1$, we compute the update
$$
r_{k+1} \quad:= \quad v^{(k+1)}- v^{(k)}\quad=\quad -\nabla_v \nabla_v^T\, \ell( v^{(k)})
\cdot \nabla \ell( v^{(k)}).
$$
The gradient and the Hessian are derived from (\ref{dirder}) and  (\ref{dirhess}) by taking $A=G_{i}$ and $B=G_{j}$.

\begin{figure}[t!]
\centering
\subfigure[$n=10$]{\includegraphics[scale=0.29]{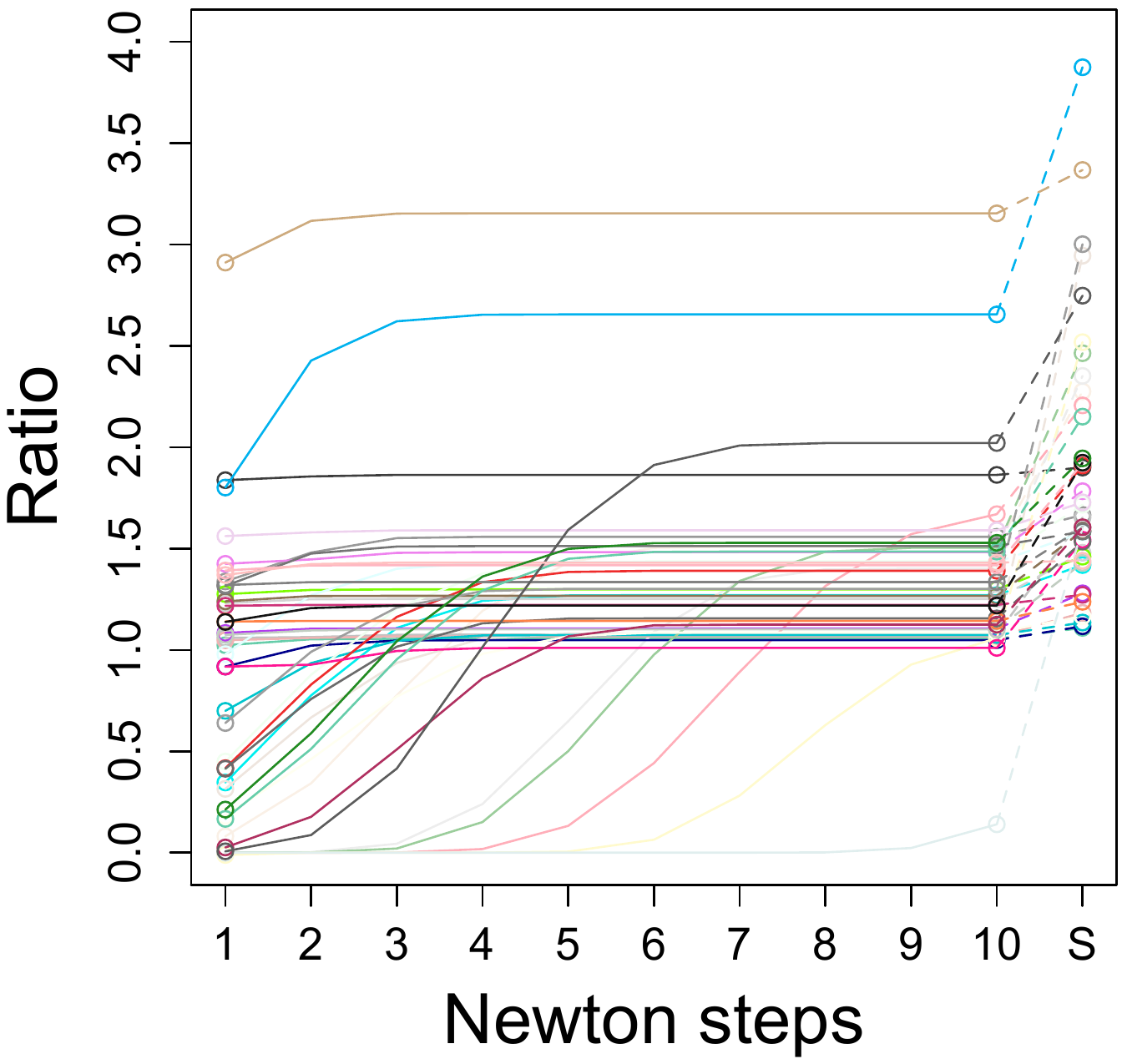}}\quad
\subfigure[$n=50$]{\includegraphics[scale=0.29]{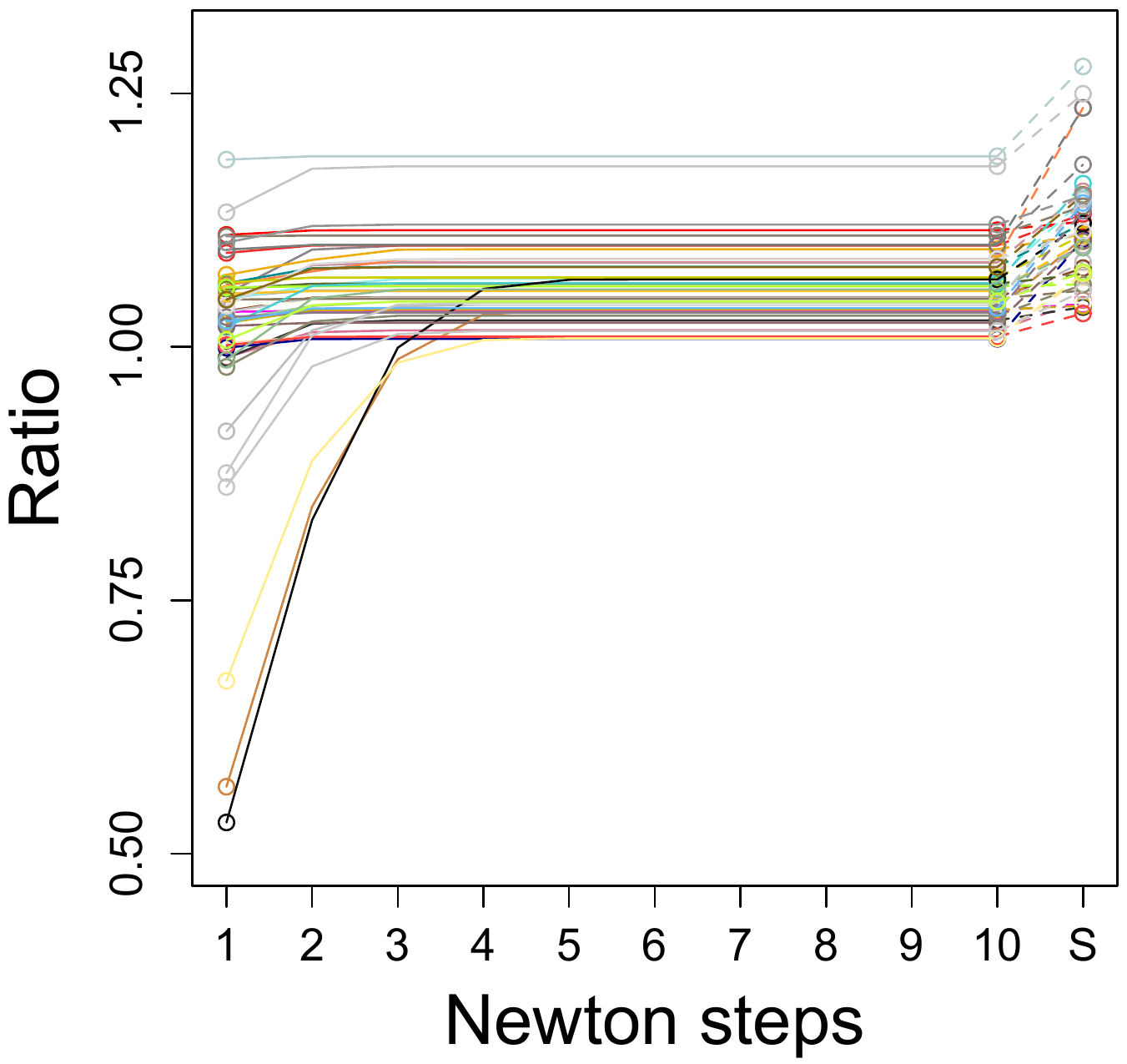}}\quad
\subfigure[$n=100$]{\includegraphics[scale=0.29]{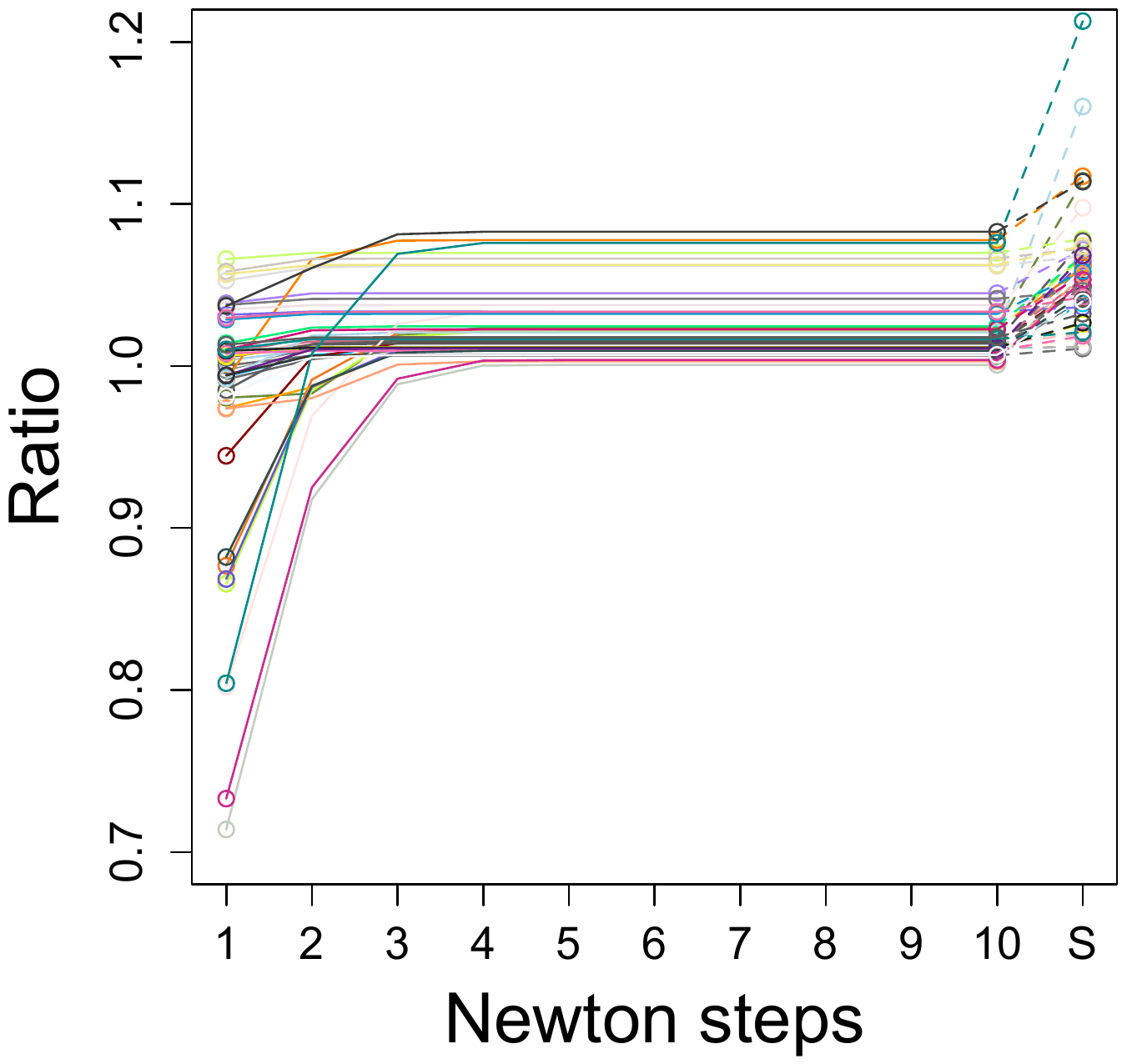}} 
\caption{Plot of 50 Newton-Raphson paths showing the ratio of the likelihood of the correlation matrix obtained by the Newton-Raphson algorithm in the first 10 steps compared to the likelihood of the true data-generating correlation matrix given in Example~\ref{ex_corr} (a). The last point represents the ratio of the likelihood of the sample covariance matrix and the true data-generating correlation matrix.}
\label{fig_corr_3}
\end{figure}

\begin{exmp}\label{ex_corr}
We show how the Newton-Raphson method performed on two examples. We sampled $n$ observations from a multivariate normal distribution with correlation matrix
$$
\textrm{(a)}\quad \Sigma_{v^*} = \begin{bmatrix} 1 & 1/2 & 1/3 \\ 1/2 & 1 & 1/4 \\ 1/3 & 1/4 & 1\end{bmatrix}, \qquad\quad 
\textrm{(b)}\quad \Sigma_{v^*} = \begin{bmatrix} 1 & 1/2 & 1/3 & 1/4\\ 1/2 & 1 & 1/5 & 1/6 \\ 1/3 & 1/5 & 1 & 1/7 \\ 1/4 & 1/6 & 1/7 & 1\end{bmatrix}.
$$
We display in Figure~\ref{fig_corr_3} the Newton-Raphson paths for the $3\times 3$ correlation matrix given in (a) and in Figure~\ref{fig_corr_4} we display the paths for the $4\times 4$ correlation matrix given in (b). Note that the scaling of the $y$-axis varies in each plot to provide better visibility of the different paths. In Figures~\ref{fig_corr_3}  and~\ref{fig_corr_4} we  plotted the ratio of the likelihood of the correlation matrix obtained by the Newton-Raphson algorithm and the likelihood of the true data-generating correlation matrix. We show the first 10 steps of the Newton-Raphson algorithm. The last point is the ratio of the likelihood of the sample covariance matrix and the true data-generating covariance matrix. If the MLE $\Sigma_{\hat{v}}$ exists, then the following inequalities hold:
$$
1\;\leq\; \frac{\textrm{likelihood}(\Sigma_{\hat{v}})}{\textrm{likelihood}(\Sigma_{v^*})}\;\leq\; \frac{\textrm{likelihood}(S_n)}{\textrm{likelihood}(\Sigma_{v^*})}.
$$
The first inequality follows from the fact that $\Sigma_{v^*}$ lies in the model and $\Sigma_{\hat{v}}$ maximizes the likelihood over the model, whereas the second inequality follows from the fact that $S_n$ is the unconstrained MLE of the Gaussian likelihood.
These inequalities are also evident in Figures~\ref{fig_corr_3} and \ref{fig_corr_4}. This is an indication that the Newton-Raphson algorithm has converged to the MLE. 

To produce Figures~\ref{fig_corr_3} and \ref{fig_corr_4} we performed 50 simulations for $n=10$, $n=50$, and $n=100$ and plotted those simulations for which $\Sigma_{\bar{v}}$ was not singular since otherwise the corresponding likelihood is undefined.  In Figure~\ref{table_boundary} we provide the mean and standard deviation for $\mathbb{P}(\Sigma_{\hat{v}}\nsucc 0)$ and $\mathbb{P}(2S_n -\Sigma_{\hat{v}}\nsucc 0)$, corresponding to the probability of landing outside the region $\Delta_{2S_n}$.

\begin{figure}[t!]
\centering
\subfigure[$n=10$]{\includegraphics[scale=0.29]{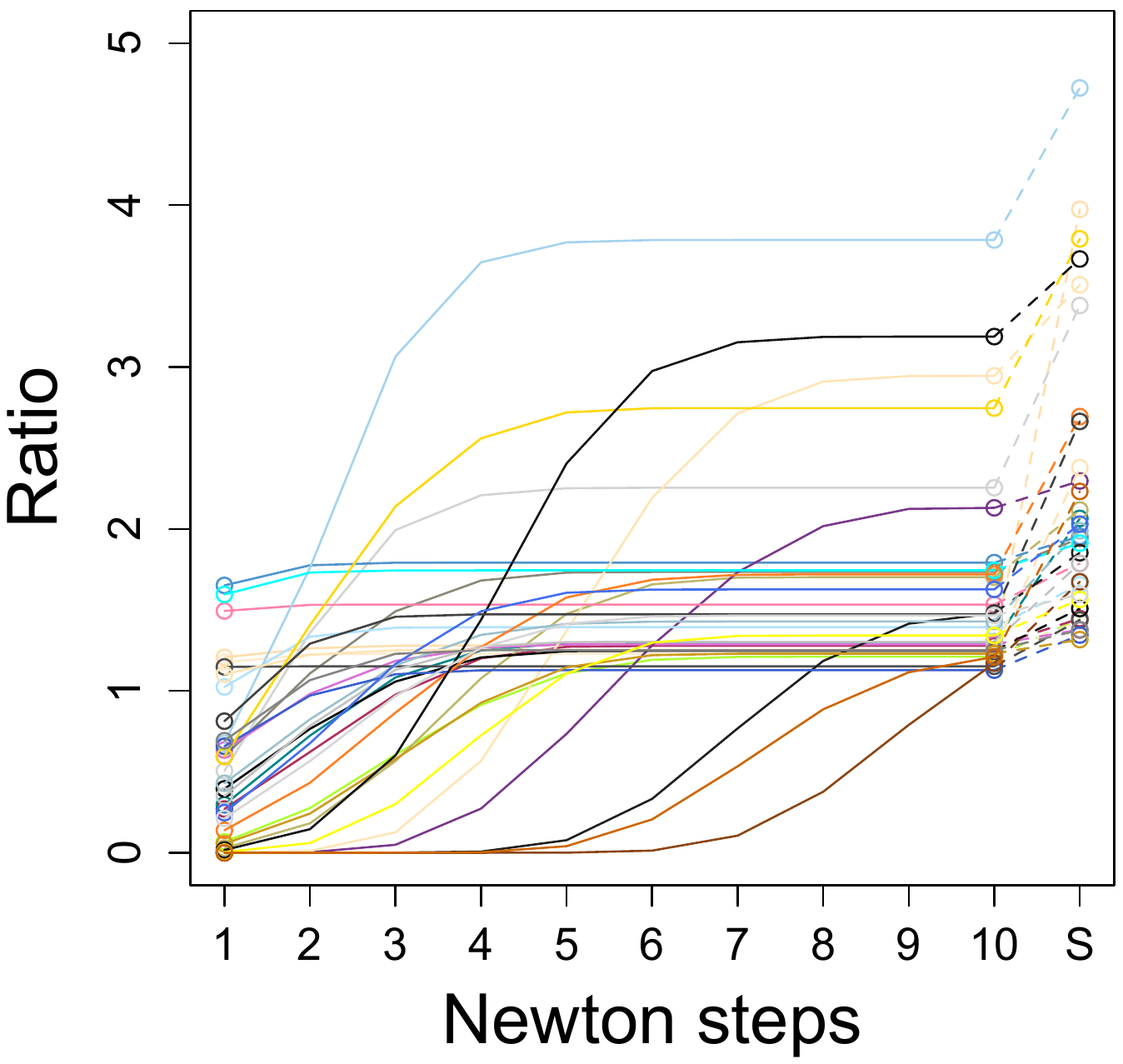}}\quad
\subfigure[$n=50$]{\includegraphics[scale=0.29]{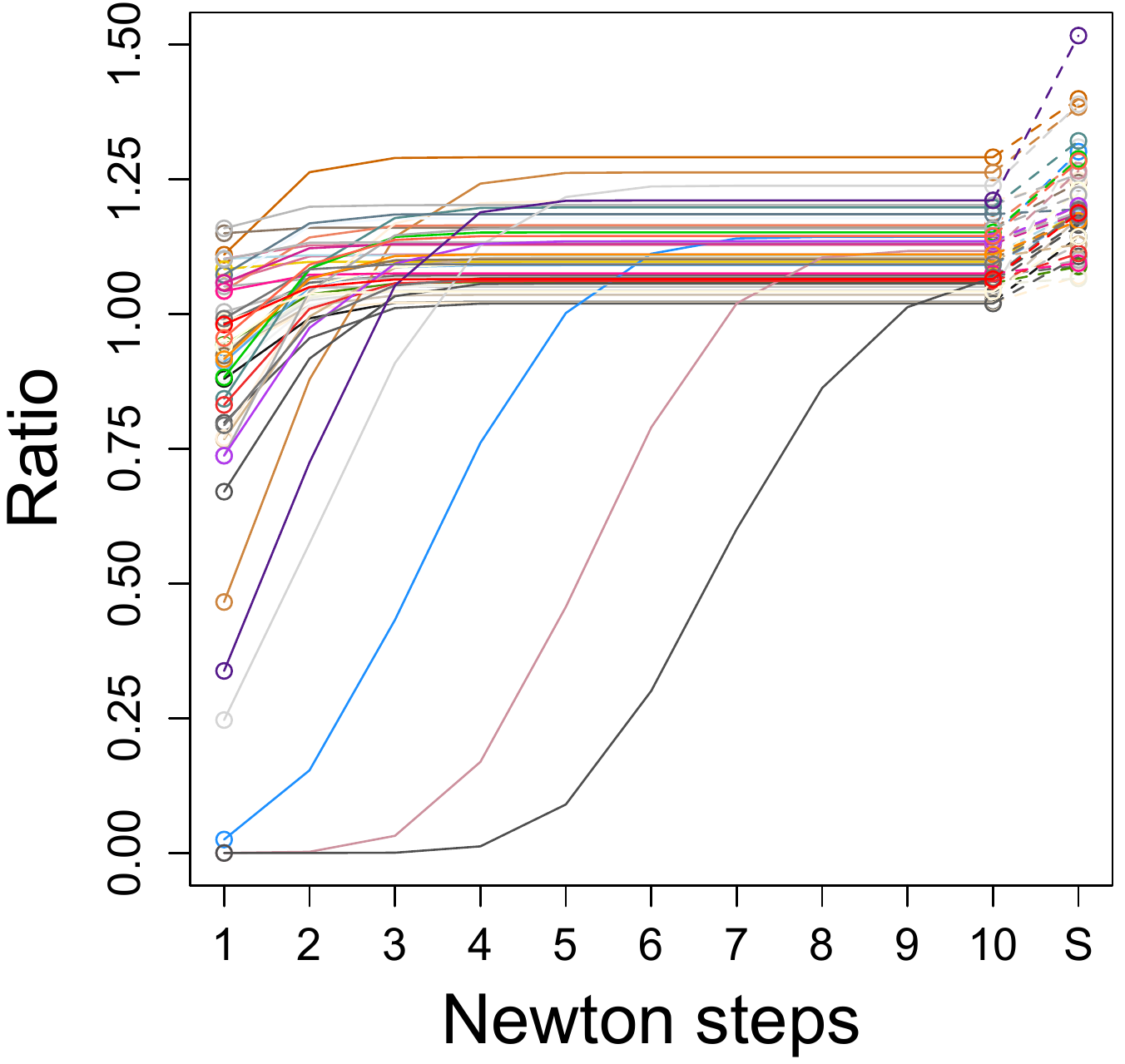}}\quad
\subfigure[$n=100$]{\includegraphics[scale=0.29]{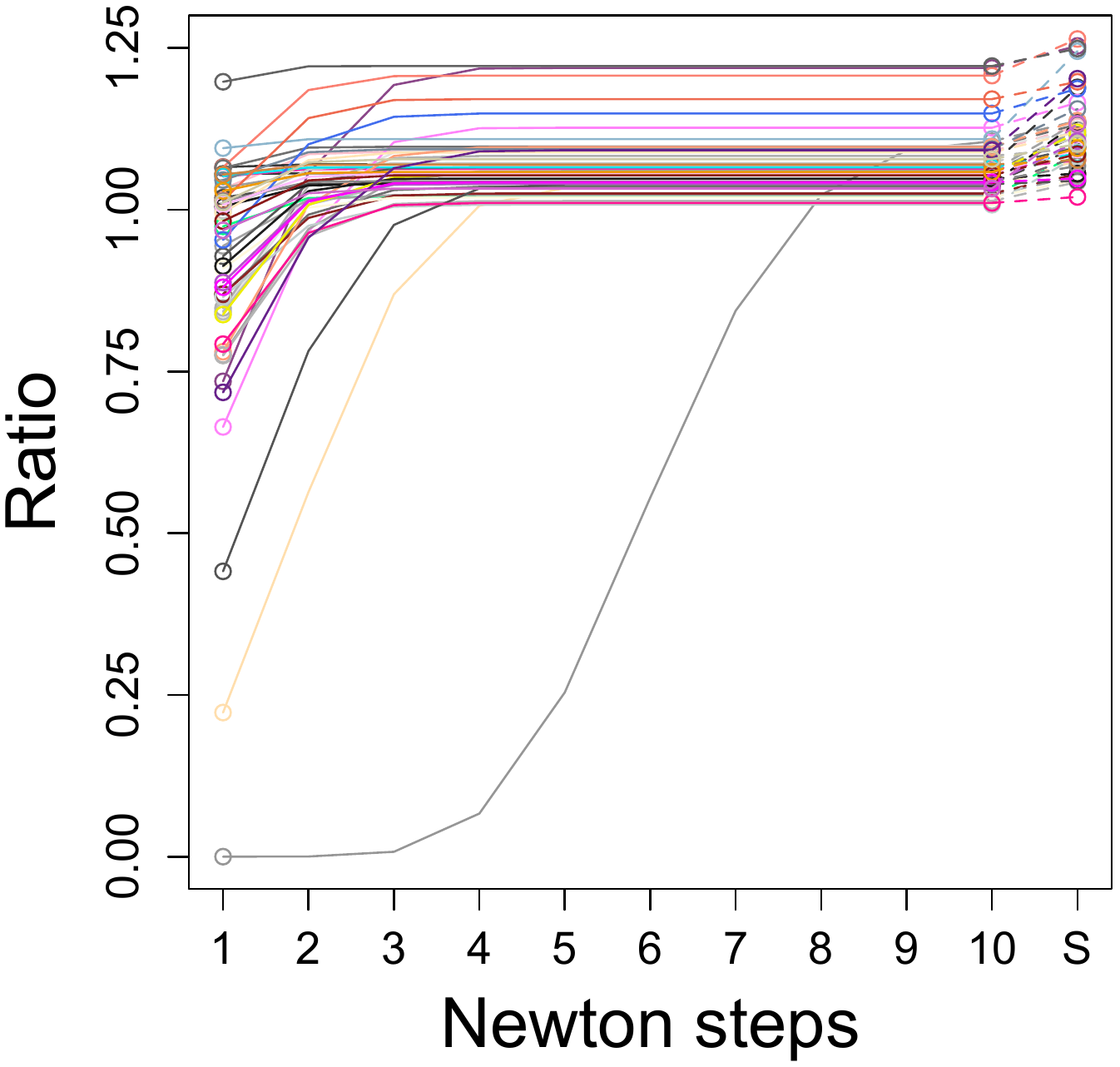}} 
\caption{Plot of 50 Newton-Raphson paths showing the ratio of the likelihood of the correlation matrix obtained by the Newton-Raphson algorithm in the first 10 steps compared to the likelihood of the true data-generating correlation matrix given in Example~\ref{ex_corr} (b). The last point represents the ratio of the likelihood of the sample covariance matrix and the true data-generating correlation matrix.}
\label{fig_corr_4}
\end{figure}

\begin{figure}[b]
\caption{Probability that the MLE is outside the region $\Delta_{2S_n}$.}
\label{table_boundary}
\centering
\vspace{0.3cm}
\begin{tabular}{l | c c  | c c |}
& \multicolumn{2}{|c|}{$3\times 3$ correlation matrix} & \multicolumn{2}{|c|}{$4\times 4$ correlation matrix} \\
& $\mathbb{P}(\Sigma_{\hat{v}}\nsucc 0)$ & $\mathbb{P}(2S_n -\Sigma_{\hat{v}}\nsucc 0)$ & $\mathbb{P}(\Sigma_{\hat{v}}\nsucc 0)$ & $\mathbb{P}(2S_n -\Sigma_{\hat{v}}\nsucc 0)$ \\ \hline
$n=10$ & $0.493$ ($\pm 0.046$) & $0.120$ ($\pm 0.030$)& $0.693$ ($\pm 0.044$)&  $0.077$ ($\pm 0.024$)\\
$n=50$ & $0.037$ ($\pm 0.017$) & $0.003$ ($\pm 0.005$)& $0.055$ ($\pm 0.022$)& $0.004$ ($\pm 0.006$)\\
$n=100$ & $0.003$ ($\pm 0.005$) & $0.000$ ($\pm 0.000$)& $0.006$ ($\pm 0.007$) & $0.000$ ($\pm 0.000$)
\end{tabular}
\end{figure}

In Figures ~\ref{fig_corr_3} and \ref{fig_corr_4} we see that the Newton-Raphson method converges, in about $3$ steps, for estimating $3\times 3$ correlation matrices with $100$ observations. The same procedure takes slightly longer for fewer observations or for $4\times 4$ matrices, and in all scenarios the Newton-Raphson method appears to converge in fewer than $10$ steps.
\end{exmp}

\section{Choice of the estimator and violation of Gaussianity}

In this section we present additional statistical analyses that support the use of the MLE of the covariance matrix in models with linear constraints on the covariance matrix. We first show that the MLE compares favorably to the least squares estimator with respect to various loss functions. We then show that the MLE is robust with respect to violation of the Gaussian assumption.

\subsection{Comparison of the MLE and the least squares estimator}\label{sec_4}

In this section, we analyze through simulations the comparative behavior of the MLE and the least squares estimator, with respect to various loss functions. We argue that for linear Gaussian covariance models the MLE is usually a better estimator than the least squares estimator. One reason is that, especially for small sample sizes, $\Sigma_{\bar{v}}$ can be negative definite, whereas $\Sigma_{\hat{v}}$ -- if it exists -- is always positive semidefinite. Furthermore, as we show in the following simulation study, even when $\Sigma_{\bar{v}}$ is positive definite, the MLE usually has smaller loss compared to $\Sigma_{\bar{v}}$. 

We analyze the following four loss functions:
\begin{enumerate}
\item[(a)] the $\ell_{\infty}$-loss: $\norm{\hat{\Sigma} - \Sigma^*}_{\infty}$,
\item[(b)] the Frobenius loss: $\norm{\hat{\Sigma} - \Sigma^*}_F$,
\item[(c)] the quadratic loss: $\norm{\hat{\Sigma}{\Sigma^*}^{-1} - \mathbb{I}_p}$, and
\item[(d)] the entropy loss: $\tr(\hat{\Sigma}{\Sigma^*}^{-1})-\log\det(\hat{\Sigma}{\Sigma^*}^{-1})-p$
\end{enumerate}
The functions (a) and (b) are standard loss functions, and the loss functions (c) and (d) were proposed in~\cite{Berger94}. We study as an example the time series model for circular serial correlation coefficients discussed in~\cite[Eq.~(5.9)]{andersonLinearCovariance}. This model is generated by $G_0=0$, $G_1=\mathbb{I}_p$, and $G_2$ defined by
$$(G_2)_{ij} = \left\{ \begin{array}{ll}
 1, & \textrm{if $|i-j| = 1\ (\!\!\!\!\!\!\mod p)$}\\
 0, & \textrm{otherwise.}
  \end{array} \right.$$

We display in Figure~\ref{fig_loss} the losses resulting from each loss function above when simulating data under two time series models for circular serial correlation coefficients on 10 nodes with the true covariance matrix defined by $v^*=(0,1,0.3)$ and $v^*=(0,1,0.45)$. Note that the covariance matrix is singular for $v^*=(0,1,0.5)$. Every point in Figure~\ref{fig_loss} corresponds to 1000 simulations and we considered only those simulations for which the least squares estimator was positive semidefinite. Instances where the least squares estimator was singular only occurred for $n=100$ and $v^*=0.45$; in this case we found that $\mathbb{P}(\Sigma_{\bar{v}}\nsucc 0) = 0.011\; (\pm 0.010)$.

In Figure~\ref{fig_loss} we see that, especially for small sample sizes and when the true covariance matrix is close to being singular, the MLE has significantly smaller loss than the least squares estimator. This is to be expected in particular for the entropy loss, since this loss function seems to favor the MLE. However, it is surprising that, even with respect to the $\ell_{\infty}$-loss, the MLE compares favorably to the least squares estimator. This shows that forcing the estimator to be positive semidefinite, as is the case for the MLE, puts a constraint on all entries jointly and leads even entry-wise to an improved estimator. We show here only our simulation results for the time series model for circular serial correlation coefficients; however, we have observed similar phenomena for various other models. 

\begin{figure}[t!]
\centering
\subfigure[$v^*=(0,1,0.3)$, \newline \hspace*{1.3em} $\ell_{\infty}$-loss]{\includegraphics[scale=0.17]{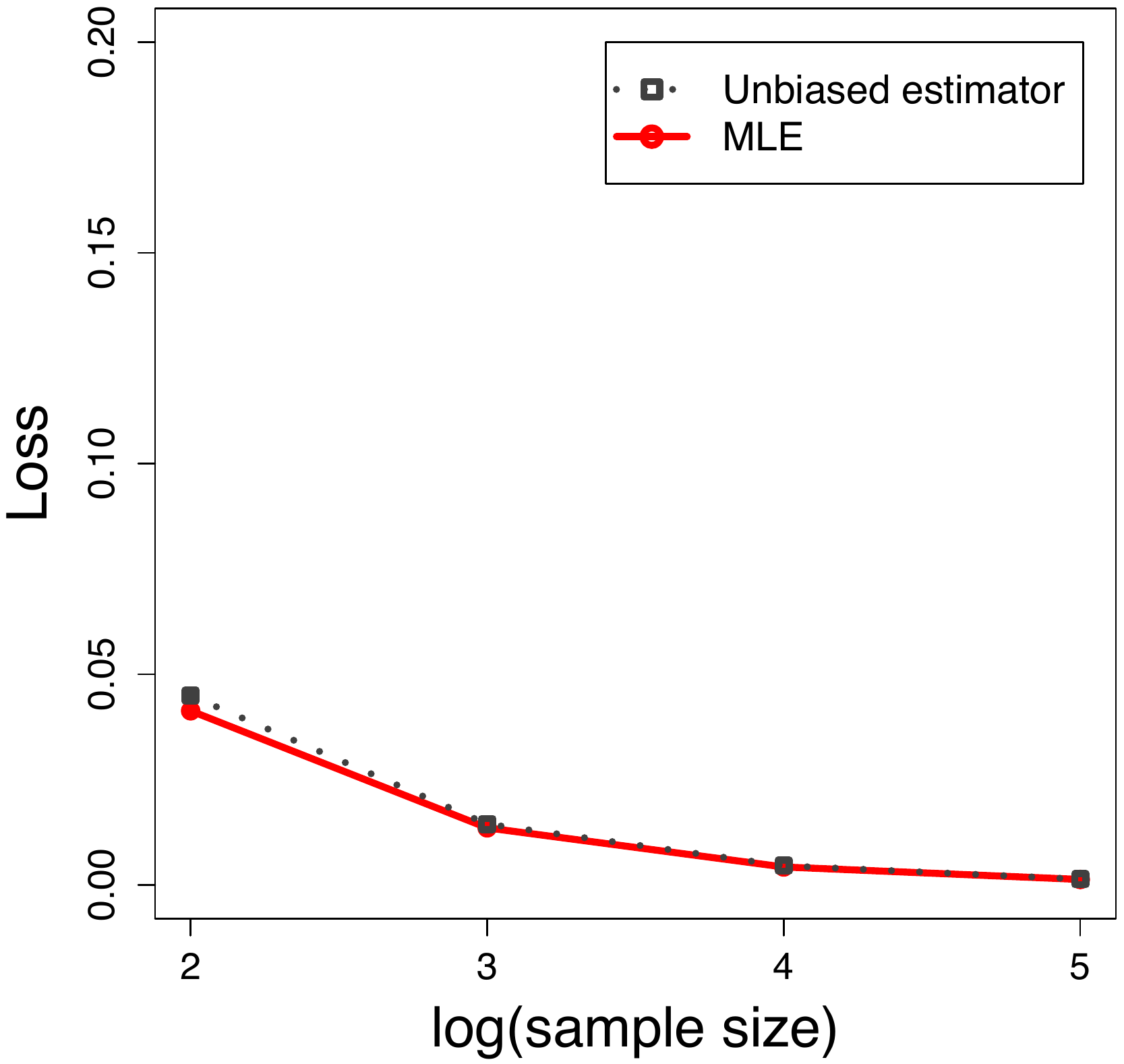}} \;
\subfigure[$v^*=(0,1,0.3)$, \newline \hspace*{1.3em} Frobenius loss]{\includegraphics[scale=0.17]{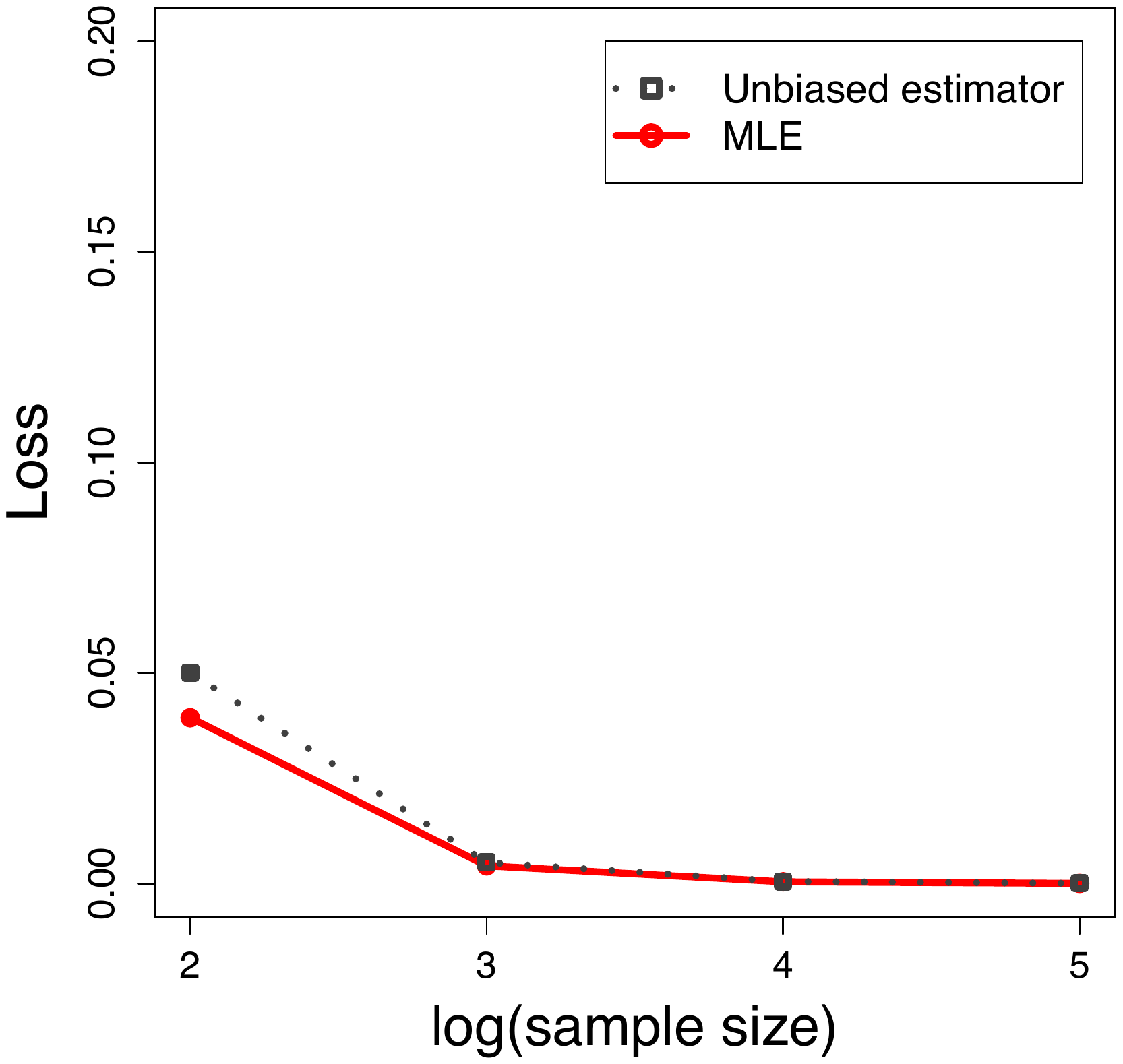}}\;
\subfigure[$v^*=(0,1,0.3)$, \newline \hspace*{1.3em} quadratic loss]{\includegraphics[scale=0.17]{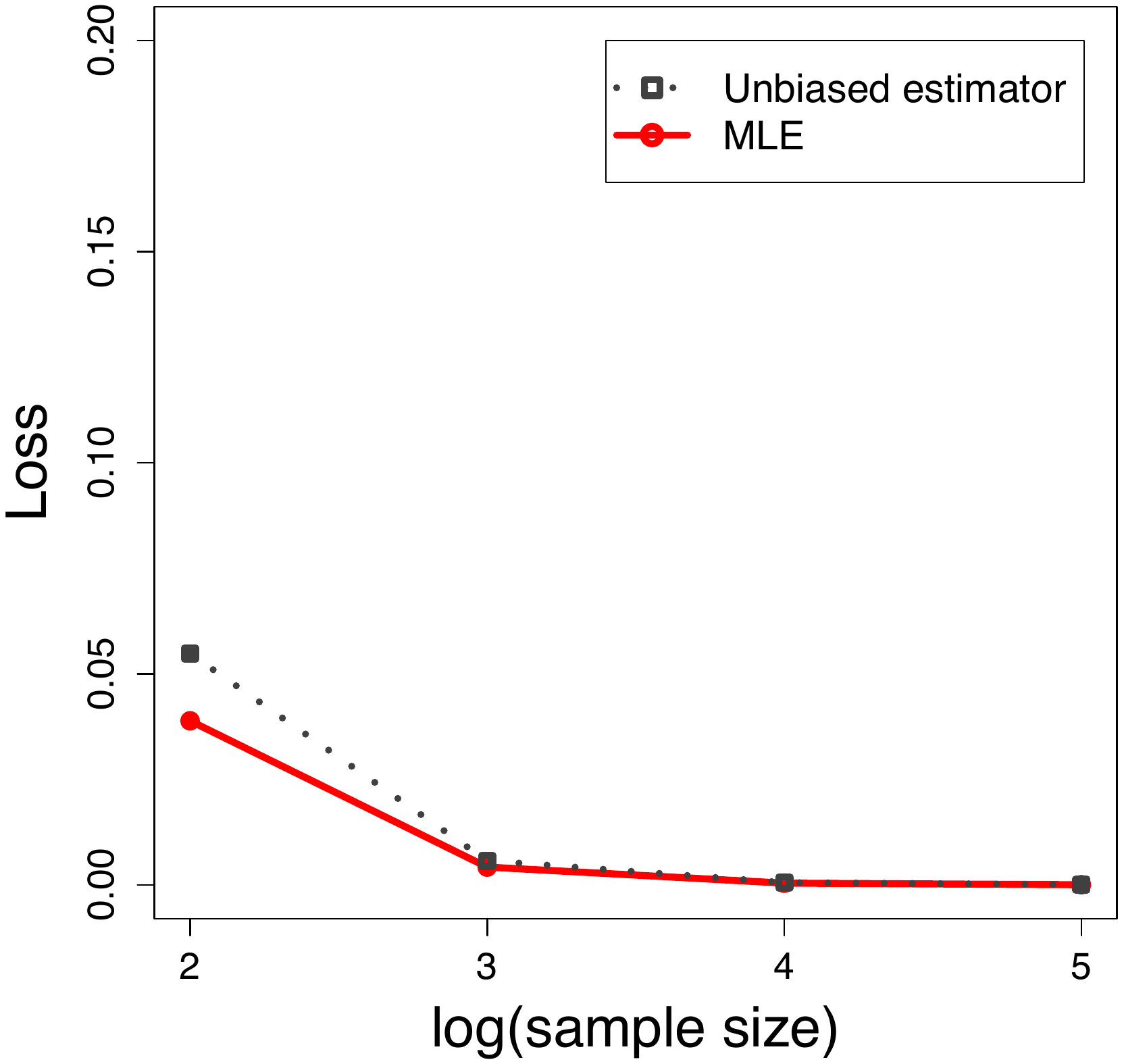}}\;
\subfigure[$v^*=(0,1,0.3)$, \newline \hspace*{1.3em} entropy loss]{\includegraphics[scale=0.17]{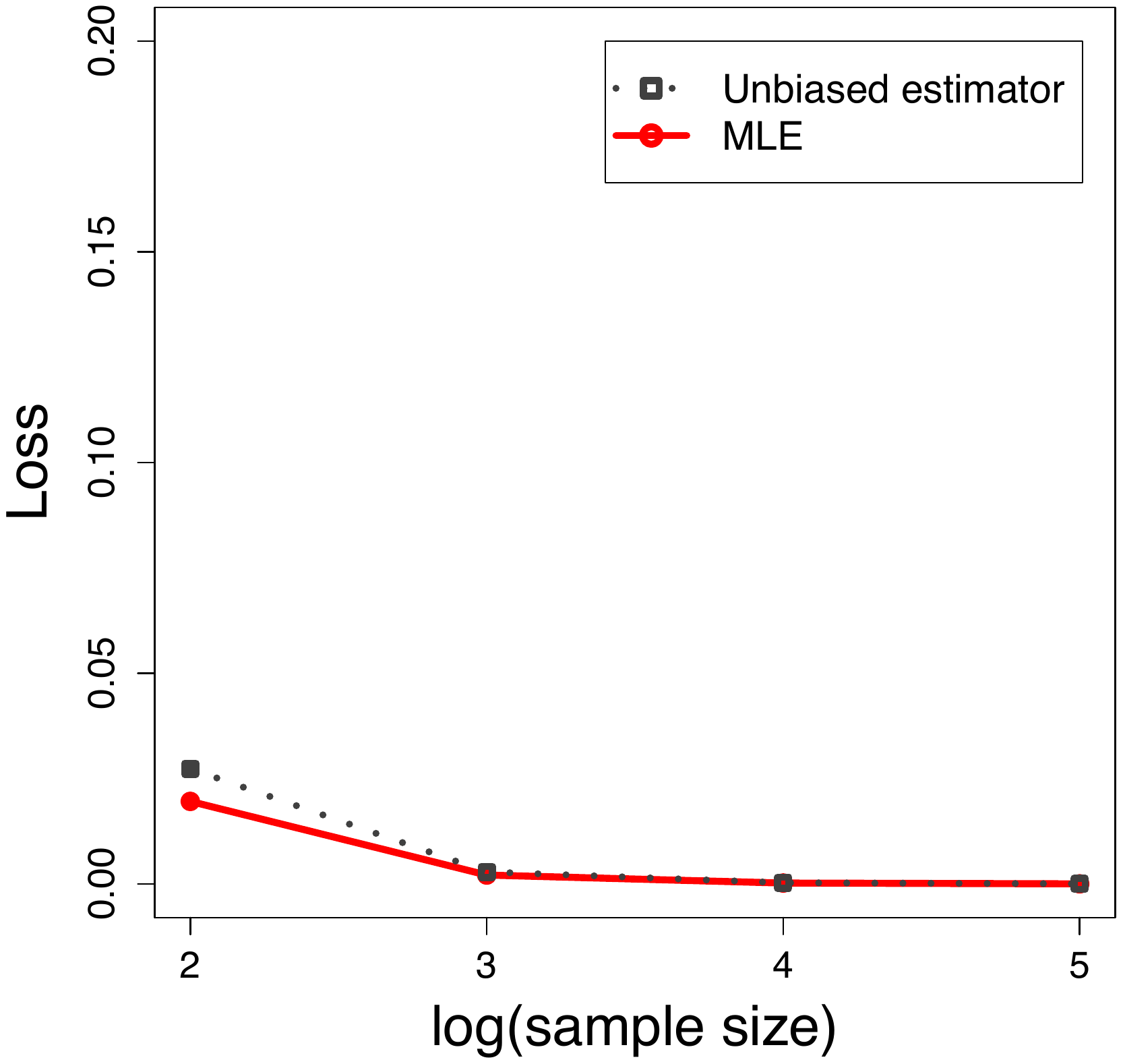}}
\vspace{0.4cm}

\subfigure[$v^*=(0,1,0.45)$, \newline \hspace*{1.3em} $\ell_{\infty}$-loss]{\includegraphics[scale=0.17]{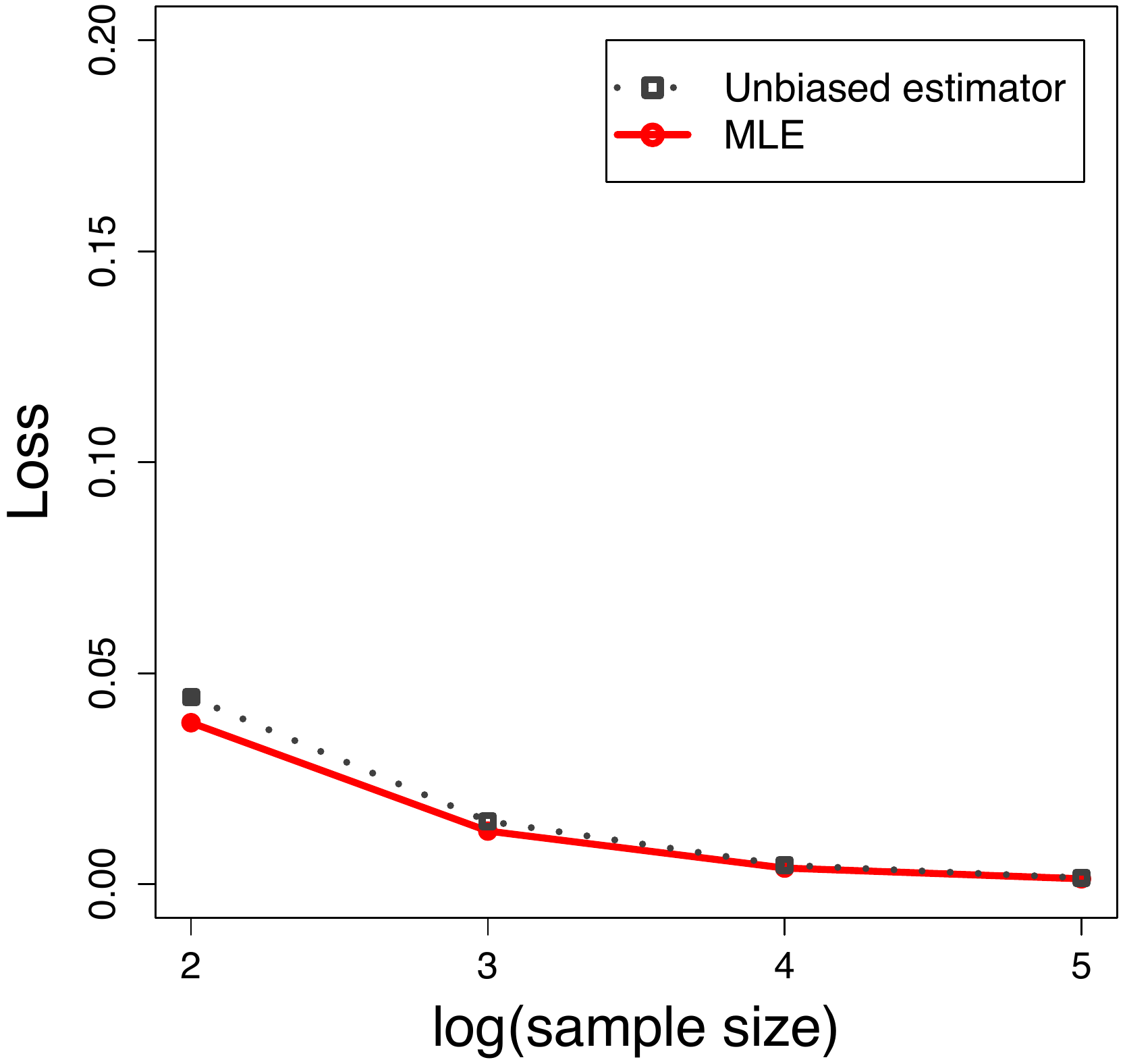}} 
\quad
\subfigure[$v^*=(0,1,0.45)$, \newline \hspace*{1.3em} Frobenius loss]{\includegraphics[scale=0.17]{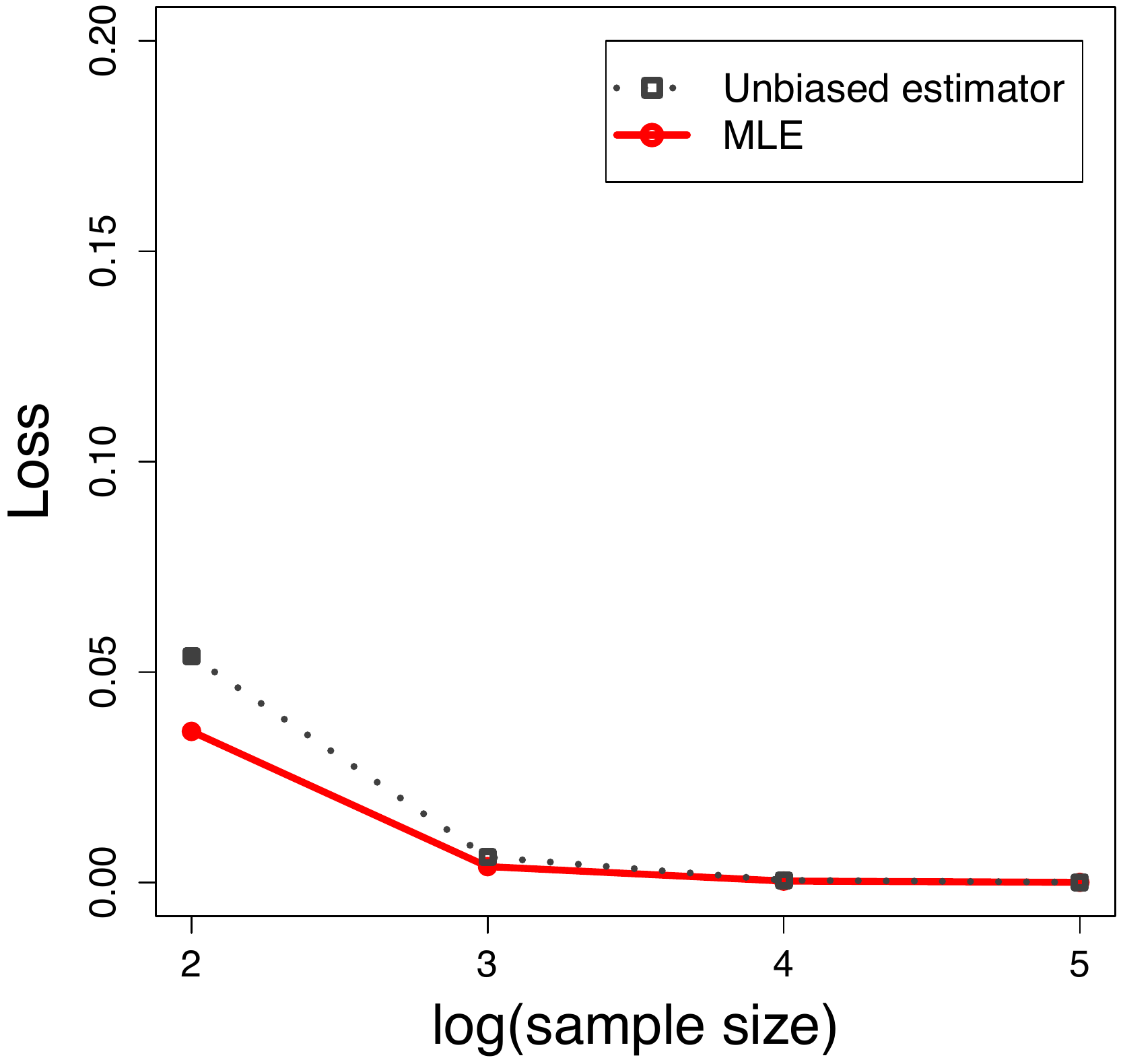}}\;
\subfigure[$v^*=(0,1,0.45)$, \newline \hspace*{1.3em} quadratic loss]{\includegraphics[scale=0.17]{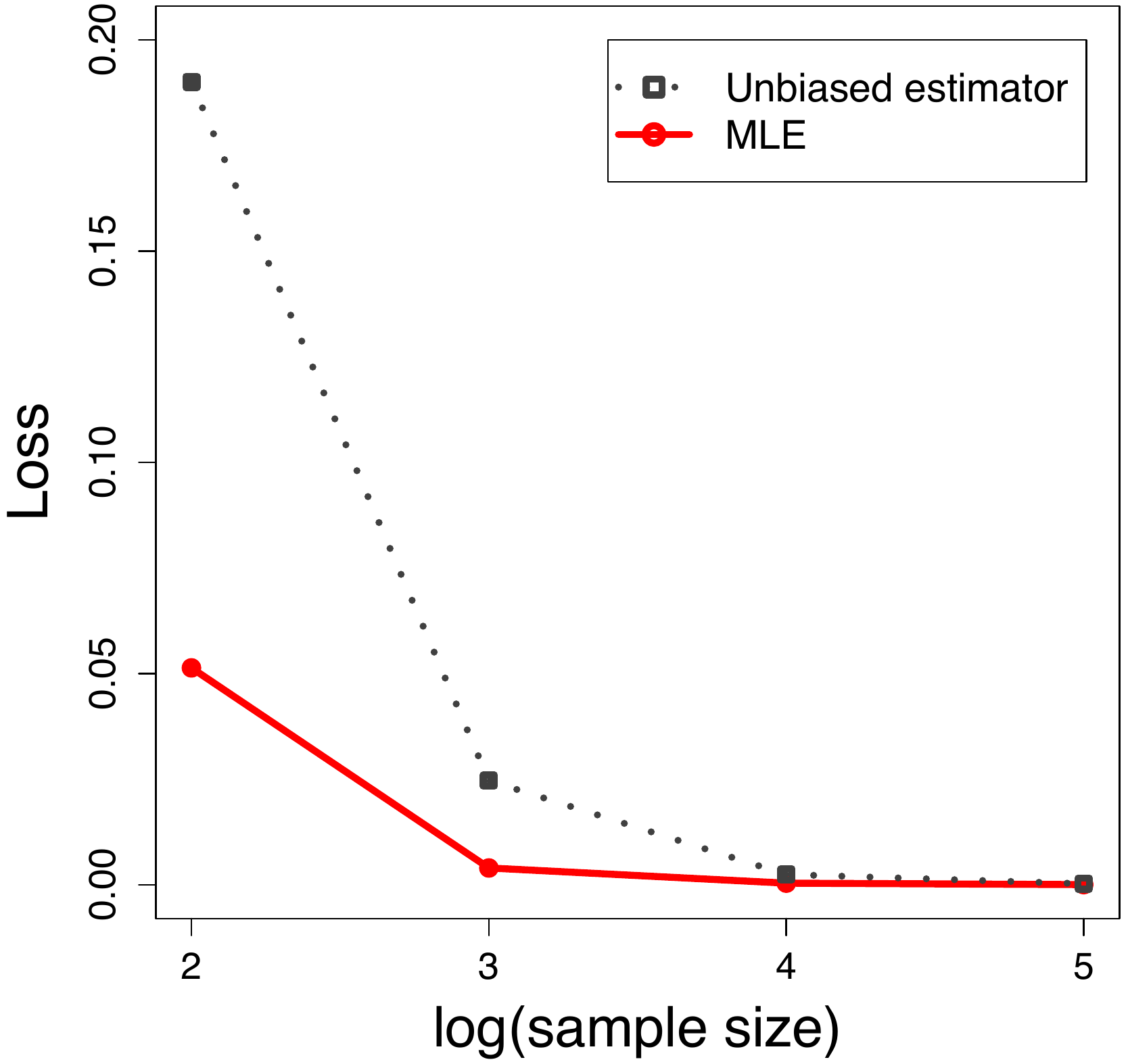}}\;
\subfigure[$v^*=(0,1,0.45)$, \newline \hspace*{1.3em} entropy loss]{\includegraphics[scale=0.17]{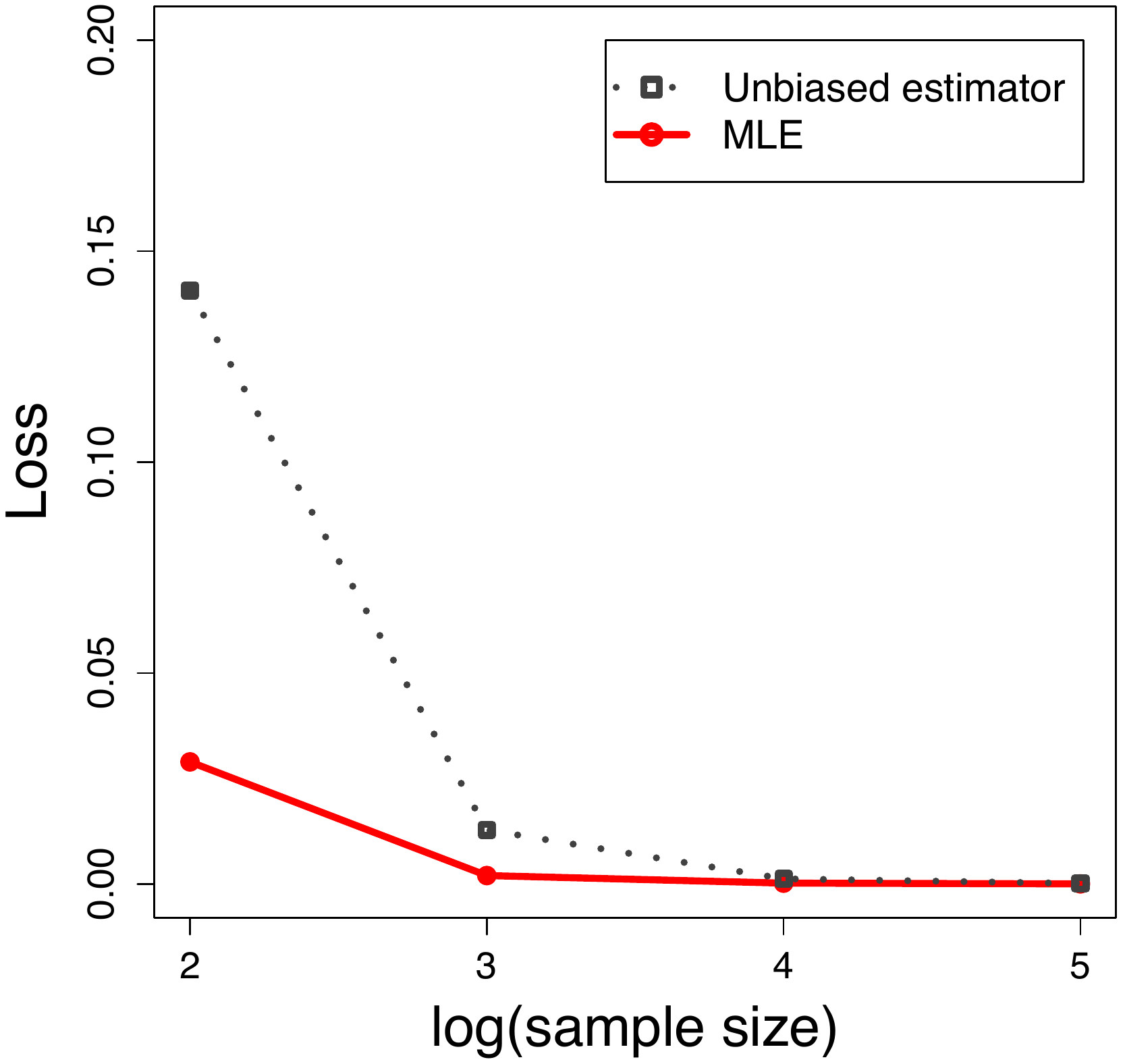}}
\vspace{0.1cm}
\caption{Comparison of MLE and least squares estimator using different loss functions for the time series model for circular serial correlation coefficients on 10 nodes defined by the true parameters being $v^*=(0,1,0.3)$ (top line) and $v^*=(0,1,0.45)$ (bottom line).}
\label{fig_loss}
\end{figure}

\subsection{Violation of Gaussianity}\label{sec:robustness}
\label{sec_Gaussian_violation}

In this section, we provide a simple analysis of model misspecification with respect to the Gaussian assumption. In our exposition we compare the results for the standard Gaussian distribution to some natural alternatives. 

Let $X$ be a $p\times n$ random matrix with i.i.d.~entries $X_{ij}$. First, suppose that the $X_{ij}$, instead of being independent standard Gaussian random variables, are sampled from a subgaussian distribution, i.e., a distribution with thinner tails than the Gaussian distribution. Then by the universality result in \cite[Theorem I.1.1]{feldheim2010}, we obtain 
$$
\P\left(\lambda_{\min}(XX^{T})>\frac{n}{2}\right)\;\;\approx\;\; F\left(\frac{\frac{n}{2}-(\sqrt{p}-\sqrt{n})^{2}}{(\sqrt{p}-\sqrt{n})\left(\frac{1}{\sqrt{p}}-\frac{1}{\sqrt{n}}\right)^{1/3}}\right),
$$
where $F$ is again the Tracy-Widom distribution. As in the Gaussian setting, the approximation is also of order $O(p^{-2/3})$~\cite{feldheim2010}. Hence, our results hold more generally for subgaussian distributions.

\begin{figure}[t!]
\centering
\subfigure[$p=3$]{\includegraphics[scale=0.23]{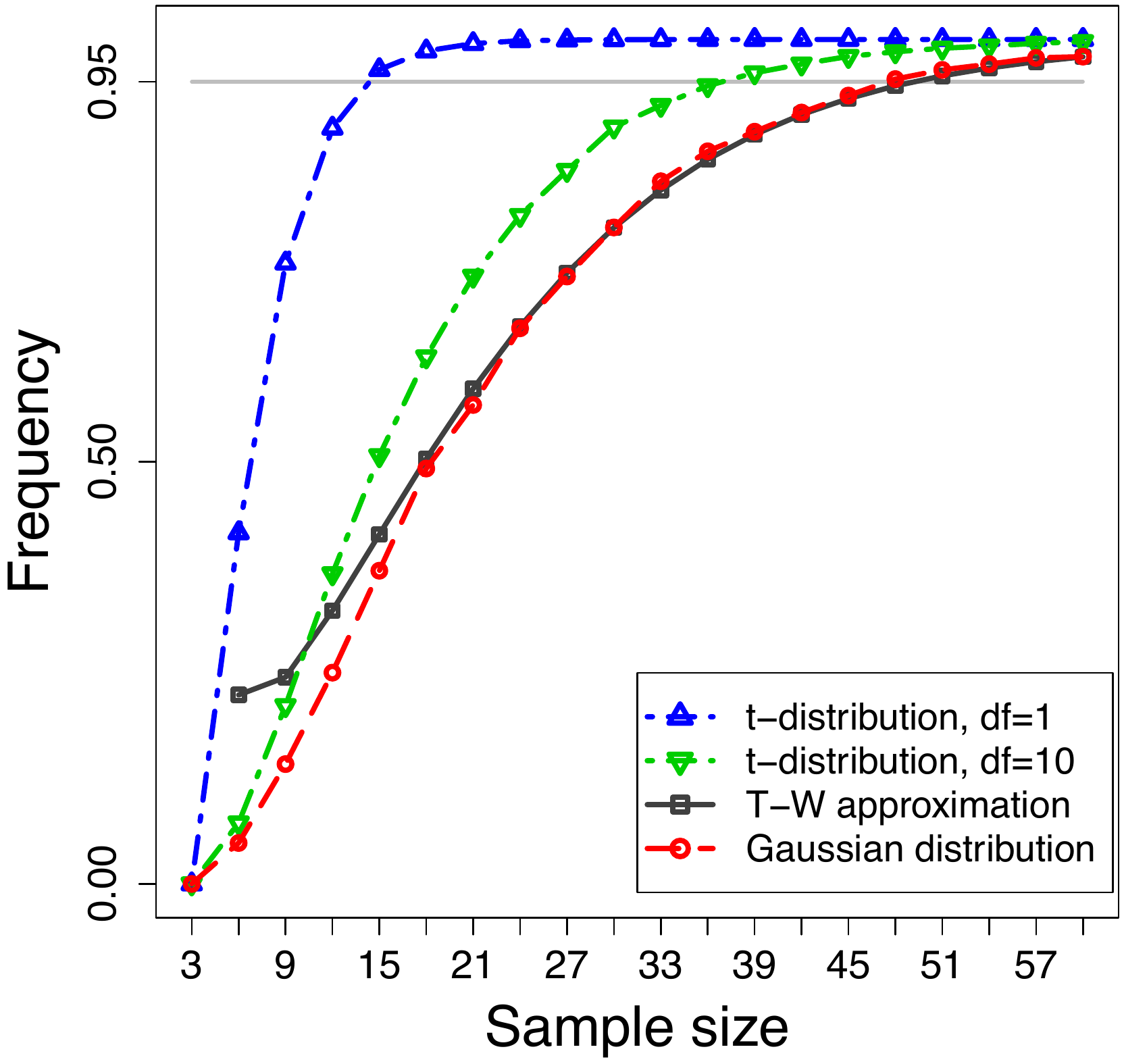}}\quad
\subfigure[$p=5$]{\includegraphics[scale=0.23]{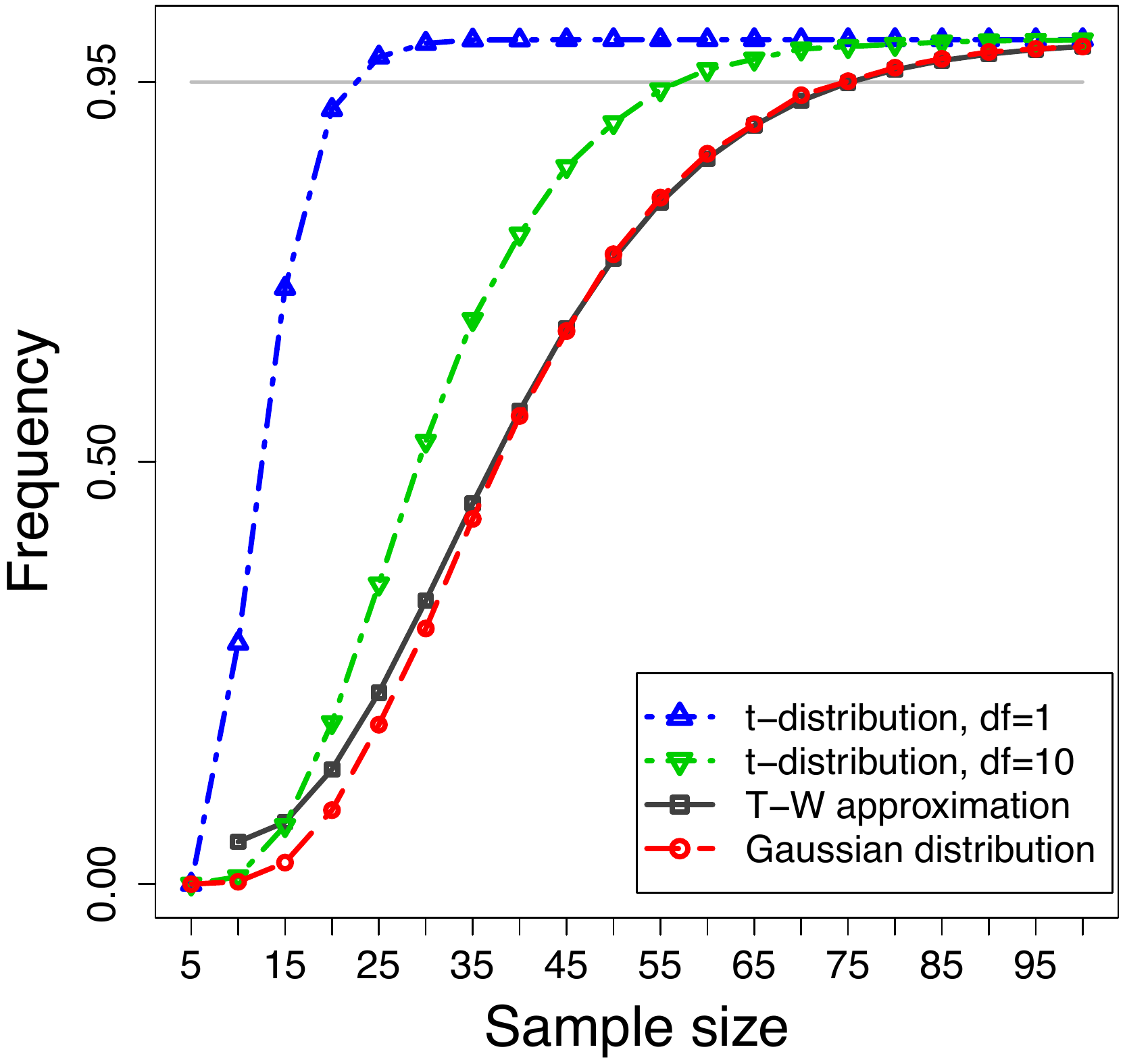}}\quad
\subfigure[$p=10$]{\includegraphics[scale=0.23]{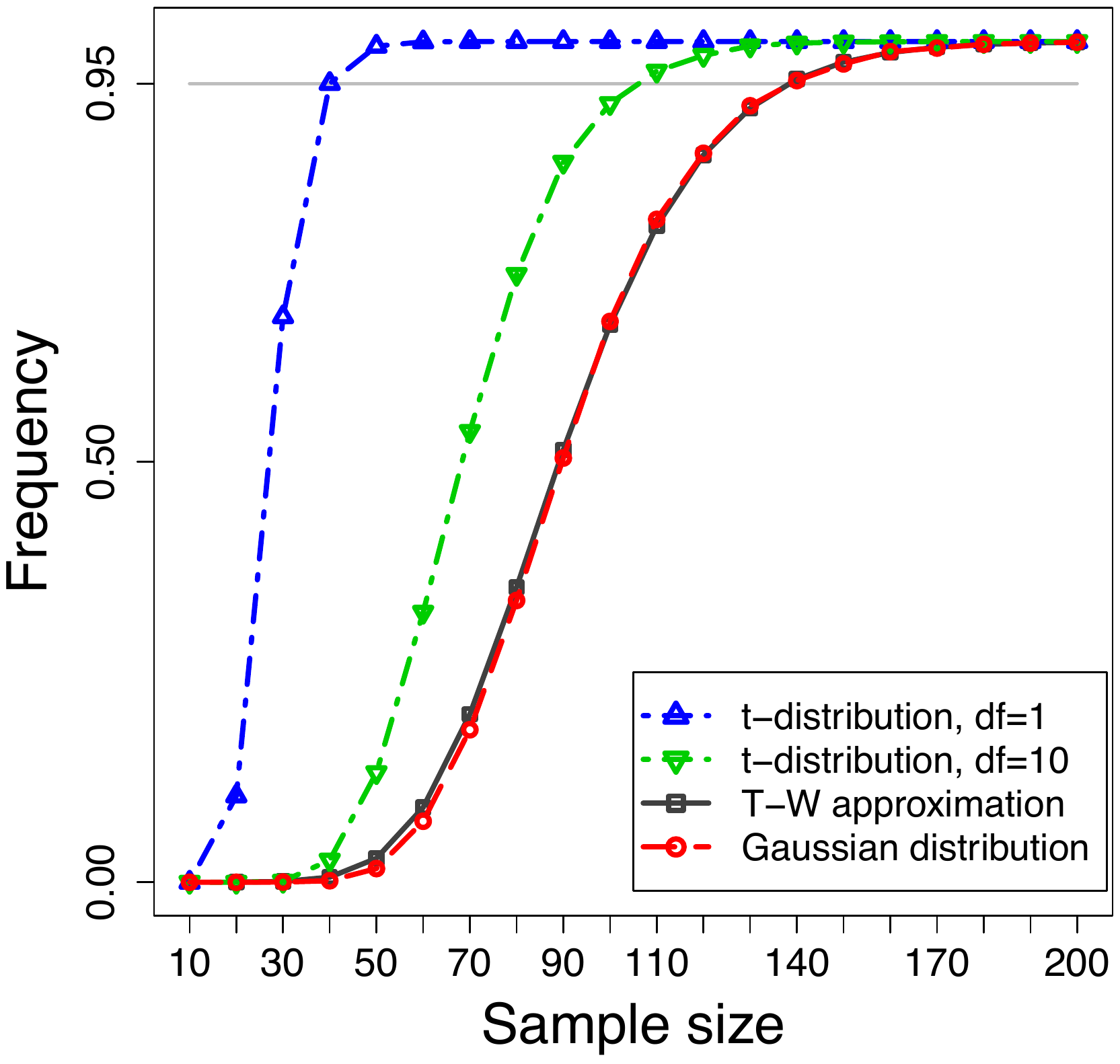}} 
\caption{Comparison of the probability that $2S_{n}^{d}-\mathbb I_{p}\succ 0$ for different values of $d\geq 1$. The Gaussian case corresponds to the limit $d\to \infty$.}
\label{fig:unb3}
\end{figure}

Next, we analyze the case in which the Gaussian assumption is violated in such a way that the true distribution has fatter tails. For this, we assume that each $X_i$ has a multivariate $t$-distribution with scale matrix $\mathbb I_{p}$ and $d$ degrees of freedom. Note that in the limit as $d\to\infty$ we obtain the original standard Gaussian setting. We denote by $S_{n}^{d}$ the sample covariance matrix corresponding to $n$ i.i.d. observations from a multivariate $t$-distribution with $d$ degrees of freedom. Note that for $d>2$, $S_n^d \to d/(d-2) \mathbb{I}_p$, almost surely, as $n\to \infty$. Therefore, as $n\to\infty$, all eigenvalues of $S_n^d$ converge to $d/(d-2)$, which is a decreasing function of $d$. Hence, for sufficiently large $n$, 
\begin{equation}\label{eq:multivariate_t_inequality}
\P(2S_{n}^{d}-\mathbb I_{p}\succ 0) > \P(2S_{n}-\mathbb I_{p}\succ 0),
\end{equation}
indicating that the optimization problem is better behaved under model misspecification with respect to the degrees of freedom in a multivariate $t$-distribution. Our simulations shown in Figure~\ref{fig:unb3} suggest that (\ref{eq:multivariate_t_inequality}) also holds for finite sample sizes. 

Moreover, as a consequence of results in~\cite{Koltchinskii}, we obtain a finite-sample exponential lower bound of the form $1-c\sqrt{p/n}$ on the minimal eigenvalue of the sample covariance matrix based on a sample of $n$ i.i.d. observations from a multivariate $t$-distribution with $d>4$ degrees of freedom. More generally, \cite[Theorem 1.3]{Koltchinskii} derived a finite-sample exponential lower bound for the minimal eigenvalue of the sample covariance matrix based on a random sample of $n$ observations from \emph{any} multivariate isotropic random vector $X$ for which there exists $\eta>2$ and $L\geq 1$ such that for every $t\in S^{p-1}$, the unit sphere in $\R^p$, and every $u > 0$, 
$$
\P(| \langle X, t\rangle |> u) \leq \frac{L}{u^{2+\eta}}.
$$

In summary, the analysis in this section shows that our results are surprisingly robust against model misspecification with respect to the Gaussian assumption for distributions with thinner or with fatter tails.

\section{Discussion}

The likelihood function for linear Gaussian covariance models is, in general, multimodal. However, as we have proved in this paper, multimodality is relevant only if the model is highly misspecified or if the sample size is not sufficiently large so as to compensate for the dimension of the model. We identified a convex region $\Delta_{2S_n}$ on which the likelihood function is strictly concave. Using recent results linking the distribution of extreme eigenvalues of the Wishart distribution to the Tracy-Widom law, we derived asymptotic conditions which guarantee when the true covariance matrix and the global maximum of the likelihood function both lie in the region $\Delta_{2S_n}$. 
Since the approximation by the Tracy-Widom law is accurate even for very small sample sizes, this makes our results useful for applications. An important consequence of our work is that:

\vspace{0.15cm}


\noindent \emph{In the case of linear models for $p\times p$ covariance matrices, where $p$ is as small as $2$, a sample size of $14 {\hskip2pt} p$ suffices for the true covariance matrix to lie in $\Delta_{2S_{n}}$ with probability at least $0.95$.
}


\vspace{0.15cm}

\noindent Moreover, this result is robust against various violations of Gaussianity. Since the goal is to estimate the true covariance matrix, the estimation process should focus on the region $\Delta_{2S_n}$, and then a boundary point that maximizes the likelihood function over $\Delta_{2S_n}$ may possibly be of more interest than the global maximum. If violations of Gaussianity are to be expected, one might want to consider alternative loss functions other than the Gaussian likelihood function. However, our results show that also in this case, the optimization problem should be performed over the convex region $\Delta_{2S_n}$, since the true data generating covariance matrix is contained in this region with high probability.

We emphasize that our results provide \emph{lower} bounds on the probabilities that the maximum likelihood estimation problem for linear Gaussian covariance models is well behaved. This is due to the fact that $\Delta_{2S_n}$ is contained in a larger region over which the likelihood function is strictly concave, and this region is contained in an even larger region over which the likelihood function is unimodal. We believe that the analysis of these larger regions will lead to many interesting research questions. 

To summarize, in this paper we showed that, similarly as for Gaussian graphical models or in fact any Gaussian model with linear constraints on the inverse covariance matrix,  the problem of computing the MLE in a Gaussian model with linear constraints on the covariance matrix is with high probability a convex optimization problem. This opens a range of questions regarding maximum likelihood estimation in Gaussian models where we impose a combination of linear constraints on the covariance and the inverse covariance matrix. Another related question is to study whether similar results hold for maximum likelihood estimation for directed Gaussian graphical models. Finally, in this paper, we concentrated on the setting where $n>p$, motivated by applications to phylogenetics. It would be of great interest to consider also the high-dimensional setting and to consider the problem of learning the underlying linear structure, both important directions for future research.

\bigskip

\appendix

\section{Concentration of measure inequalities for the Wishart distribution}
\label{appendix_Wishart}

Let $W_n$ be a white Wishart random variable, i.e., $W_n\sim\mathcal{W}_p(n, \mathbb{I}_p)$. In the following, we derive finite-sample bounds for $\mathbb{P}(|\tr(W_n)-np|\leq \epsilon)$ and for $\mathbb{P}(|\log\det W_n-p\log(n)|\leq\epsilon)$.

It is well known that $\tr(W_{n})$ has a $\chi^{2}_{np}$ distribution; see, e.g.~\cite[Theorem 3.2.20]{muirheadMVSbook}. Therefore, by Chebyshev's inequality we find that
\begin{equation}
\label{eq_trace}
\P\big(|{\rm Tr}(W_{n})-np|\leq \epsilon\big)\quad\geq\quad 1-\frac{2np}{\epsilon^{2}}.
\end{equation}

Finite-sample bounds that are more accurate than (\ref{eq_trace}) can be obtained using results in  \cite{inglot2006asymptotic, Inglot2010, laurent_chisq}. We present here the bounds from~\cite{laurent_chisq}. Let $X\sim\chi^2_d$; using the original formulation in \cite[page 1325]{laurent_chisq}, we obtain for any positive $x$, 
\begin{align*}
\P(X-d\geq 2\sqrt{dx}+2x)&\;\leq\; \exp(-x),\\
\P(d-X\geq 2\sqrt{dx})&\;\leq\; \exp(-x).
\end{align*}
This can be rewritten for $\epsilon>0$ as
\begin{equation*}
\begin{aligned}
\P(X-d\geq \epsilon)&\;\leq\; \exp\Big(-\frac{1}{2} \big(d+\epsilon-\sqrt{d(d+2\epsilon)}\:\big)\Big),\\
\P(d-X\geq \epsilon)&\;\leq\; \exp\left(-\frac{\epsilon^{2}}{4d}\right).
\end{aligned}
\end{equation*}
Consequently, we obtain the following result:

\begin{prop}
\label{prop_trace}
Let $W_n\sim\mathcal{W}_p(n, \mathbb{I}_p)$. Then, 
$$
\P(|\tr(W_{n})-np|\leq \epsilon) 
\geq\; 1-\exp\left(-\frac{1}{2}\left(np+\epsilon-\sqrt{np(np+2\epsilon)}\right)\right)-\exp\left(-\frac{\epsilon^{2}}{4np}\right).
$$
\end{prop}

We now derive finite-sample bounds for $\log\det(W_n)$. In order to apply Chebyshev's inequality, we calculate the mean and variance of $\log\det(W_n)$.  First, we calculate the moment-generating function of $\log\det(W_n)$:
\begin{align*}
M(t) &:= E \exp(t\log\det(W_n)) \\
&= E(\det(W_n)^t) = E \prod_{i=1}^p Y_i^t = \prod_{i=1}^p E(Y_i^t),
\end{align*}
where $Y_i \sim \chi^2_{n+1-i}$ and the $Y_i$ are mutually independent.  It is well-known that  
$$
E(Y_i^t) = \frac{2^t\, \Gamma(t+\tfrac12(n+1-i))}{\Gamma(\tfrac12(n+1-i))},
$$
$t > -\tfrac12$, and in that case we obtain 
$$
M(t) = \prod_{i=1}^p \frac{2^t \Gamma(t+\tfrac12(n+1-i))}{\Gamma(\tfrac12(n+1-i))}.
$$
Hence, 
$$
\log M(t) = pt\log 2 + \sum_{i=1}^p \log\Gamma(t+\tfrac12(n+1-i)) - \log\Gamma(\tfrac12(n+1-i)).
$$
Differentiating with respect to $t$ and then setting $t=0$, we obtain 
$$
E(\log\det(W_n)) = \frac{{\rm d}}{{\rm d}t} \log M(t)\Big|_{t=0} = p\log 2 + \sum_{i=1}^p \psi(\tfrac12(n+1-i)),
$$
where $\psi(t) = {\rm d}\log\Gamma(t)/{\rm d}t$ is the digamma function.  Differentiating a second time and then setting $t=0$, we obtain 
$$
{\rm{Var}}(\log\det(W_n)) = \frac{{\rm d}^2}{{\rm d}t^2} \log M(t)\Big|_{t=0} = \sum_{i=1}^p \psi_1(\tfrac12(n+1-i)),
$$
where $\psi_1(t) = {\rm d}^2\log\Gamma(t)/{\rm d}t^2$ is the trigamma function.  Therefore, by the non-central version of Chebyshev's inequality, we have
\begin{equation}\label{eq_det}
\mathbb{P}(|\log\det W_n- p\log(n))|\leq \epsilon) 
\;\ge\; 1 - \frac{\sum\limits_{i=1}^p \psi_1(\frac{n+1-i}{2}) +\Big(p\log\left(\frac{2}{n}\right)+\sum\limits_{i=1}^p \psi(\frac{n+1-i}{2})\Big)^2}{\epsilon^2}.
\end{equation}

\vspace{0.5cm}
 
\section*{Acknowledgments}
We thank Mathias Drton, Noureddine El Karoui, Robert D.~Nowak, Lior Pachter, Philippe Rigollet, and Ming Yuan for helpful discussions. We also thank two referees, an Associate Editor, and an Editor for constructive comments.

P.Z.'s research is supported by the European Union 7th Framework Programme (PIOF-GA-2011-300975). CU's research is supported by the Austrian Science Fund (FWF) Y 903-N35. This project was started while C.U. was visiting the Simons Institute at UC Berkeley for the program on ``Theoretical Foundations of Big Data'' during the fall 2013 semester. D.R.'s research was partially supported by the U.S.~National Science Foundation grant DMS-1309808 and by a Romberg Guest Professorship at the Heidelberg University Graduate School for Mathematical and Computational Methods in the Sciences, funded by German Universities Excellence Initiative grant GSC 220/2.


\bibliographystyle{alpha}
\bibliography{algebraic_statistics}

\end{document}